\documentclass[preprint]{elsarticle}
\usepackage{amsmath,amssymb}
\usepackage{amsthm}
\usepackage{graphicx}
\theoremstyle{plain}
\newtheorem{thm}{Theorem}
\newtheorem{prop}[thm]{Proposition}
\begin{document}

\begin{frontmatter}

	\title{Implicit-explicit schemes for compressible Cahn-Hilliard-Navier-Stokes equations on staggered grids}
	
	\author[UV1]{Andreu Martorell}
	\ead{andreu.martorell@uv.es}
	\author[UV1]{Pep Mulet}
	\ead{mulet@uv.es}
	\author[UV1]{Dionisio F. Y\'a\~nez}
	\ead{dionisio.yanez@uv.es}
	\address[UV1]{Departament de Matem\`atiques.  Universitat de Val\`encia (EG) (Spain)}
	
	\begin{abstract}
        We propose a second-order implicit-explicit (IMEX) time-stepping scheme for the isentropic, compressible Cahn-Hilliard-Navier-Stokes equations discretized on staggered (MAC) grids.
        The scheme is based on finite difference approximations that ensure a stable coupling among the velocity, density and phase field, with symmetric operators acting on the discretized viscosity terms.
        
        Standard explicit methods suffer from severe time-step restrictions due to the presence of second to fourth-order diffusion terms introduced by the Cahn-Hilliard and Navier-Stokes operators.
        To overcome these challenges, we develop an IMEX Runge-Kutta scheme that treats the stiff terms implicitly while the convective terms are dealt with explicitly, with the advantage that only linear systems are solved at each stage.
        
        Numerical experiments are performed to verify the stability, accuracy and efficiency of the proposed approach.
	\end{abstract}
	
	
\end{frontmatter}

\section{Introduction}

In fluid dynamics, the Cahn-Hilliard-Navier-Stokes (CHNS) equations describe the behavior of two-phase flows with diffuse interfaces.
The Cahn-Hilliard (CH) equation captures the evolution of the phase separation process \cite{CahnHilliard59}, while the Navier-Stokes equations govern the fluid motion, accounting for viscosities and external forces like gravity.
The CHNS model has been successfully applied to a wide range of phenomena, including separation of immiscible fluids, bubble dynamics or sedimentation of colloidal suspensions \cite{Kynch52,Siano79}.

The CHNS equations admit both incompressible and compressible formulations,
each suitable for different physical scenarios.
In this work, however, we focus on the isentropic compressible version derived in \cite{AbelsFeireisl08}.

The numerical treatment of the compressible CHNS equations presents several challenges. 
The system consists on spatial partial differential operators of up to fourth-order, which restricts explicit schemes to time steps of size $\Delta t\approx\mathcal{O}(\Delta x^4)$, having large eigenvalues in the discretized operators, where $\Delta x$ is the step size.
Additionally, for ODE systems of the form $\mathbf{z}^\prime=f(\mathbf{z})$, the presence of negative-definite Jacobians with large eigenvalues allows implicit methods to take larger time steps than explicit methods, namely $\Delta t\approx\mathcal{O}(|\lambda|^{-1})$, with $|\lambda|$ denoting the largest absolute value of the Jacobian.

The inherent gradient flow structure introduced by the CH equation, whose energy functional is non-convex, complicates the stability.
In \cite{Eyre98,Vollmayr-Lee-Rutenberg2003}, an unconditionally gradient stable IMEX scheme for the CH equation was proposed by splitting the energy functional into a difference of two convex functions.
By treating the contractive part implicitly and the expansive part explicitly, the scheme was shown to be unconditionally stable.

Standard finite differences schemes for compressible flows often use collocated grids, where all the variables share the same spatial locations.
This arrangement can lead to numerical instabilities, especially when approximating high-order first derivatives due to the discretization of coupled variables at coincident points. 
In incompressible flows, it may also cause the classical checkerboard pressure-velocity decoupling \cite{harlow_welch_65,RhieChow83}.

One effective solution is to employ staggered meshes \cite{harlow_welch_65,Patankar18}, where velocity components are defined at different locations than scalar variables. 
This approach allows spatial derivatives to be approximated using centered differences without interpolation,
preventing artificial oscillations in the solution.
Moreover, staggered discretizations provide a more compact representation of advective and diffusive fluxes, 
and often preserve desirable properties such as symmetry and positive definiteness of diffusive operators, 
which are not always guaranteed on collocated grids.

Recent advances on the CHNS system include the work of \cite{AbelsFeireisl08}, 
where the existence of global-in-time-weak solutions for the compressible case is proven.
Several numerical methods have also been proposed for both compressible and incompressible versions.
For example, \cite{DhaouadiDumbserGavrilyuk25} introduces the CH equation in a first-order hyperbolic reformulation, 
while various energy-stable and mass-conservative schemes among other properties have been proposed in \cite{ChenZhao20,HanWang15,HeShi20,Jia20,LiXu21,mulet_24}.

In \cite{mulet_24}, a linearly IMEX finite difference scheme on collocated grids is analyzed for the isentropic, compressible CHNS equations. 
This approach employs mixed boundary conditions: 
homogeneous Dirichlet conditions for the velocity field and homogeneous Neumann conditions for both the phase parameter and the chemical potential.
When combined with a collocated finite-difference discretization, 
this choice results in non-symmetric discrete viscous operators for the velocity field, 
which may negatively affect the stability and robustness of the scheme, 
particularly in the equations strongly coupled to the momentum.

Numerically, these effects can produce reduced accuracy in the density and phase field variables.
In particular, violations of the expected bounds of the phase parameter were reported in \cite{mulet_24}.
One possible solution is to impose a more restrictive CFL time-step stability constraint when necessary. 
However, this strategy increases the computational cost, particularly in the 
presence of near-vacuum regions for the density or in strongly phase-separation regimes for the order parameter.

The goal of the present work is to design an efficient second-order linearly IMEX Runge-Kutta based on finite differences scheme on staggered meshes for the isentropic, compressible CHNS equations.
By treating the convective terms explicitly and the diffusive terms implicitly, the resulting method satisfies a time-step restriction governed only by the convective subsystem, i.e., $\Delta t\approx\mathcal{O}(\Delta x)$.
For the spatial discretization, we employ Mark and Cell (MAC) grids \cite{harlow_welch_65}: velocity components are located at cell-faces and density and phase field variables are stored at cell-centers.
With this staggered layout, the resulting discrete operators, including the viscous terms, are symmetric and positive definite.
Numerical results confirm the advantages of such approach: without decreasing the CFL number, it significantly reduces spurious oscillations compared with collocated grids, 
maintains non-negativity of the density, and preserves the physical bounds of the order parameter with only occasional minor violations.

The structure of the work is as follows. 
Section \ref{section_CHNS} introduces the isentropic, compressible CHNS equations in two dimensions.
In Section \ref{section_numerical_schemes} we present the spatial discretization based on a MAC approach together with the partitioned IMEX Runge-Kutta time integrator.
Section \ref{section_numerical_experiments} analyzes the performance of the proposed scheme ensuring stability, accuracy, and efficiency.
Finally, Section \ref{section_conlusion} summarizes the main contributions of the work and discusses further directions of our research.

\section{The isetropic compressible Cahn-Hilliard-Navier-Stokes Equations}\label{section_CHNS}
\subsection{Model Description}

In the present work, we study a variant of the model introduced in \cite{AbelsFeireisl08}.
Consider a binary fluid with mass concentrations $c_1$ and $c_2$, total density $\rho$, and bulk velocity $\mathbf{v}$.
Define the order parameter $c=c_1-c_2$ which distinguishes between different phases of the mixture.

Let $\Omega$ be an open subset of $\mathbb{R}^d, d=1,2,3$, and let $\varepsilon>0$ denote the diffuse-interface thickness parameter.
The Helmholtz free energy of the system is given by
$$\mathcal{E}(\rho, c)=\int_\Omega\!\left(\rho f(\rho, c) + \frac{\varepsilon}{2}|\nabla c|^2\right)\ d\mathbf{x}.$$
In an isentropic regime, the specific free energy $f$ is composed of one-phase fluid free energy $f_e$ and a double-well potential $\psi$.
Specifically, we define
\begin{equation*}
    f(\rho, c) = f_e(\rho) + \psi(c),\quad
    f_e(\rho)=C_p\frac{\rho^{\gamma-1}}{\gamma-1},\quad
    \psi(c) = \frac{1}{4}\left(c^2-1\right)^2,
\end{equation*}
where $C_p>0$ and $\gamma>1$ is the adiabatic exponent.
The Gibbs relation yields the pressure formula
$$p(\rho)=\rho^2\frac{\partial f(\rho, c)}{\partial\rho} = C_p\rho^\gamma.$$

Under these assumptions, the isentropic compressible CHNS with gravitational effects takes the form
\begin{equation}\label{compressible_CHNS}
  \left\{
  \begin{split}
    &\rho_t + \operatorname{div}(\rho \mathbf{v}) = 0,\\[2pt]
    &(\rho\mathbf{v})_t + \operatorname{div}(\rho\mathbf{v}\otimes\mathbf{v}) + \nabla p = \rho\mathbf{g} + \operatorname{div}\mathbb{T},\\[2pt]
    &(\rho c)_t + \operatorname{div}(\rho c\mathbf{v}) = \Delta\mu.
  \end{split}
  \right.
\end{equation}
The first and the second equations represent the conservation of mass and the balance of momentum taking into account the gravitation acceleration $\mathbf{g}$.
The third equation is a modified CH equation accounting for phase separation.
The stress tensor $\mathbb{T}$ includes both viscous and capillary contributions, namely,
\begin{equation*}
    \mathbb{T} = 
    \nu(\nabla\mathbf{v} + \nabla\mathbf{v}^T) + 
    \lambda\operatorname{div}\mathbf{v}\mathbb{I} + \frac{\varepsilon}{2}|\nabla c|^2\mathbb{I} - \varepsilon(\nabla c\otimes\nabla c),
\end{equation*}
where $\lambda,\nu>0$ denote the viscosity coefficients.
The generalized chemical potential is given by 
$$\mu = \frac{1}{\rho}\frac{\delta\mathcal{E}}{\delta c} = \frac{\partial f}{\partial c} - \frac{\varepsilon}{\rho}\Delta c.$$
The initial and boundary conditions are set to 
\begin{equation}\label{bdry_cond}
    \begin{split}
        &\rho(0,x) = \rho_0(x),\quad \mathbf{v}(0,x) = \mathbf{v}_0(x),\quad c(0,x) = c_0(x),\\
        &\mathbf{v}\left.\right|_{\partial\Omega} =
         \nabla c\cdot\mathbf{n}\left.\right|_{\partial\Omega} =
         \nabla \mu\cdot\mathbf{n}\left.\right|_{\partial\Omega} = 0,
    \end{split}
\end{equation}
where $\mathbf{n}$ denotes the outward unit normal vector to the boundary $\partial\Omega$.

For $\gamma > \frac{3}{2}$ and suitable initial data $(\rho_0, \mathbf{v}_0, c_0)$, the system \eqref{compressible_CHNS} supplemented with \eqref{bdry_cond} admits global-in-time weak solutions in the sense of Di Perna and Lions; see \cite[Theorem 1.2]{AbelsFeireisl08}. 

In what follows, we restrict our analysis to the two-dimensional version of \eqref{compressible_CHNS}. 
Denoting the velocity by $\mathbf{v} = (v_1, v_2)$ and assuming that
gravity acts only in the $y$ dimension, i.e., $\mathbf{g}=(0,g)$, the system becomes
\begin{subequations}\label{eq_compressible_chns_2D}
    \begin{align}
    \begin{aligned}
        \rho_t + (\rho v_1)_x + (\rho v_2)_y &= 0,
    \end{aligned}
    \label{eq_compressible_chns_2D_conser_mass}
    \\[8pt]
    \begin{aligned}
        (\rho v_1)_t + (\rho v_1^2 + p(\rho))_x + (\rho v_1 v_2)_y
        &= \tfrac{\varepsilon}{2}(c_y^2 - c_x^2)_x - \varepsilon(c_x c_y)_y \\
        &\quad + \nu \Delta v_1 + (\nu + \lambda)(v_{1,xx} + v_{2,xy}),
    \end{aligned}
    \label{eq_compressible_chns_2D_mom_x}
    \\[8pt]
    \begin{aligned}
        (\rho v_2)_t + (\rho v_2^2 + p(\rho))_y + (\rho v_1 v_2)_x
        &= \rho g + \tfrac{\varepsilon}{2}(c_x^2 - c_y^2)_y - \varepsilon(c_x c_y)_x \\
        &\quad + \nu \Delta v_2 + (\nu + \lambda)(v_{1,xy} + v_{2,yy}),
    \end{aligned}
    \label{eq_compressible_chns_2D_mom_y}
    \\[8pt]
    \begin{aligned}
        (\rho c)_t + (\rho c v_1)_x + (\rho c v_2)_y
        &= \Delta\!\left( \psi'(c) - \frac{\varepsilon}{\rho}\Delta c \right).
    \end{aligned}
    \label{eq_compressible_chns_2D_ch}
    \end{align}
\end{subequations}

\section{Numerical Schemes}\label{section_numerical_schemes}

\subsection{Spatial Semidiscretization}

The purpose of this section is to design a finite difference scheme on a MAC grid \cite{harlow_welch_65} for approximating the solution of \eqref{eq_compressible_chns_2D} over the spatial domain $\Omega=(0,1)^2$.

Let $\mathbf{x}_{i,j} = (x_i,y_j)$ denote a generic point on the staggered mesh, and $h=\frac{1}{M}$ be the spatial step size.
The primal grid, where the variables $\rho$ and $c$ are defined, is formed by the $M^2$ cell-center nodes
\begin{equation*}
    \mathbf{x}_{i,j}=\left(\left(i-\tfrac{1}{2}\right)h, \left(j-\tfrac{1}{2}\right)h\right),
    \quad 
    i,j=1,\cdots,M.
\end{equation*}
The dual grids, which store the velocity components, contain $M(M-1)$ nodes. 
For the horizontal velocity $v_1$, the grid points are 
\begin{equation*}
  \mathbf{x}_{i+\frac{1}{2},j} =
    (ih,\left(j-\tfrac{1}{2}\right)h), 
    \quad 
    i=1,\ldots,M,\enspace j=1,\ldots,M-1.
\end{equation*}
For the vertical velocity $v_2$, the corresponding nodes are 
\begin{equation*}
  {\mathbf{x}_{i, j+\tfrac{1}{2}}=(\left(i-\tfrac{1}{2}\right)h,  jh)},
    \quad
    i=1,\ldots,M-1,\enspace j=1,\ldots,M.
\end{equation*}
With this setup, equations \eqref{eq_compressible_chns_2D_conser_mass} and \eqref{eq_compressible_chns_2D_ch} are discretized on primal grids, while \eqref{eq_compressible_chns_2D_mom_x} and \eqref{eq_compressible_chns_2D_mom_y} at horizontal and vertical dual grids, respectively. 

To evaluate the momentum $\mathbf{m} = (m_1,m_2)=(\rho v_1, \rho v_2)$, we compute the density at staggered points using local averages of the density at primal points, i.e.,
\begin{equation*}
    \rho_{i+\frac12,j} = \frac{\rho_{i,j} + \rho_{i+1,j}}{2},
    \quad
    \rho_{i,j+\frac12} = \frac{\rho_{i,j} + \rho_{i,j+1}}{2},
\end{equation*}
running $i$, $j$ over their respective grid indexes.
Accordingly, the momentum components $m_k=\rho v_k, k=1,2$ are defined as 
\begin{equation*}
    m_{1,i+\frac12,j} = \rho_{i+\frac12,j}v_{1,i+\frac12,j},
    \quad
    m_{2,i,j+\frac12} = \rho_{i,j+\frac12}v_{2,i,j+\frac12}.
\end{equation*}
We define the matrices $u_{k}(t)$ for $k=1,\ldots,4$ and let $U(t) = (u_k(t))_{k=1}^4$ where
\begin{equation*}
    u_{1,i,j} = \rho_{i,j},
    \quad 
    u_{2,i+\frac12,j} = m_{1,i+\frac12,j}, 
    \quad 
    u_{3,i,j+\frac12} = m_{2,i,j+\frac12},
    \quad 
    u_{4,i,j} = q_{i,j},
\end{equation*}
with $q = \rho c$, and for $i,j$ running over their respective grid indices. 
For simplicity, we will henceforth omit the explicit dependence on the temporal variable $t$ in $U$.
With this notation, the spatially discretized CHNS system \eqref{eq_compressible_chns_2D} reduces to solve $N=2M^2+2M(M-1)$ ordinary differential equations
\begin{equation}\label{eq_semi_discrete_system}
    U^\prime = \mathcal{L}(U),
\end{equation}
where the nonlinear operator $\mathcal{L}(U)$ contains the spatial discretizations of the differential operators in \eqref{eq_compressible_chns_2D}.
Specifically, we write
$$\mathcal{L}(U)= {\mathcal{C}}(U) +{\mathcal{L}}_1(U)+{\mathcal{L}}_2(U)+{\mathcal{L}}_3(U)+{\mathcal{L}}_4(U),$$ 
where the convective terms are: 
\begin{equation*}
    \begin{split}
        {\mathcal{C}}(U)_{1,i,j}&\approx -(      (\rho v_1)_x+(\rho v_2)_{y})(\mathbf{x}_{i,j}, t),\\
        {\mathcal{C}}(U)_{2,i+\frac12,j}&\approx -((\rho v_1^2 + p(\rho))_{x}+(\rho v_1 v_2)_{y})(\mathbf{x}_{i+\frac12,j}, t),\\              
        {\mathcal{C}}(U)_{3,i,j+\frac12}&\approx -((\rho v_1v_2)_{x}+(\rho v_2^2+p(\rho))_{y})(\mathbf{x}_{i,j+\frac12}, t),  \\
        {\mathcal{C}}(U)_{4,i,j}&\approx -(       (\rho c v_1)_{x}+(\rho c v_2)_{y})(\mathbf{x}_{i,j}, t),
    \end{split}
\end{equation*}
and the nonzero blocks of the diffusive terms are:
\begin{equation*}
    \begin{split}
        {\mathcal{L}}_1(U)_{3,i,j+\frac12}
            &\approx g\rho(\mathbf{x}_{i,j+\frac12},t),\\
        {\mathcal{L}}_2(U)_{2,i+\frac12,j}&\approx 
            \varepsilon(
            \frac12(c_{y}^2)_{x}-\frac12 (c_{x}^2)_x- (c_{x}c_y)_{y})(\mathbf{x}_{i+\frac12,j}, t),\\
        {\mathcal{L}}_2(U)_{3,i,j+\frac12}&\approx 
            \varepsilon(      \frac12(c_{x}^2)_{y}-\frac12 (c_{y}^2)_y    -(c_xc_y)_x)(\mathbf{x}_{i,j+\frac12}, t),\\
        {\mathcal{L}}_{3}(U)_{4,i,j}
            &\approx\Delta(\psi'(c) - \frac{\varepsilon}{\rho}  \Delta c)  (\mathbf{x}_{i,j}, t),\\
        {\mathcal{L}}_4(U)_{2,i+\frac12,j}&\approx
            (\nu(  (v_1)_{xx}+(v_1)_{yy})+(\nu+\lambda)(  (v_1)_{xx}+(v_2)_{xy}))(\mathbf{x}_{i+\frac12,j}, t),\\
        {\mathcal{L}}_4(U)_{3,i,j+\frac12}&\approx
            (\nu((v_2)_{xx}+(v_2)_{yy})  +(\nu+\lambda)(  (v_1)_{xy}+(v_2)_{yy}))(\mathbf{x}_{i,j+\frac12}, t).
    \end{split}
\end{equation*}

\subsubsection{Basic Finite Difference Operators}

For approximating first-order derivatives on the MAC grid, we employ second-order finite difference operators at interior points and first-order otherwise satisfying the boundary conditions \eqref{bdry_cond}. 

On the primal grid, the discrete first derivative operator is represented by a central difference matrix
\begin{equation}\label{eq_primal_der}
    D_M^c = \frac{1}{2h}
    \begin{bmatrix}
        -1 & 1 & 0 & \dots & 0\\
        -1 & 0 & 1 & \dots & 0\\
        \vdots & \ddots & \ddots & \ddots & \vdots\\
        0 & \dots & -1 & 0 & 1\\
        0 & \dots & 0 & -1 & 1
    \end{bmatrix} \in \mathbb{R}^{M \times M}.
\end{equation}
On the dual grids, the derivatives are approximated using the
$M\times (M-1)$ matrices
\begin{equation}\label{eq_stag_der}
    D_M = \frac{1}{h}
    \begin{bmatrix}
        1 & 0 & 0 & \dots & 0\\
        -1 & 1 & 0 & \dots & 0\\
        \cdots & \cdots & \cdots & \cdots & \cdots\\
        0 & \dots & 0 & -1 & 1\\
        0 & \dots & 0 & 0 & -1
    \end{bmatrix}, 
    \quad
    D_M^{\ast} = \frac{1}{h}
    \begin{bmatrix}
        2 & 0 & 0 & \dots & 0\\
        -1 & 1 & 0 & \dots & 0\\
        \cdots & \cdots & \cdots & \cdots & \cdots\\
        0 & \dots & 0 & -1 & 1\\
        0 & \dots & 0 & 0 & -2
    \end{bmatrix}.
\end{equation}
We define the discrete Laplacian operator with Neumann boundary conditions acting on primal grids as 
\begin{equation*}\label{eq_lap_mat}
    L_M = \frac{1}{h^2}
    \begin{bmatrix}
        -1 & 1 & 0 & \dots & 0\\
        1 & -2 & 1 & \dots & 0\\
        \vdots & \ddots & \ddots & \ddots & \vdots\\
        0 & \dots & 1 & -2 & 1\\
        0 & \dots & 0 & 1 & -1
    \end{bmatrix}
    \in\mathbb{R}^{M\times M}.
\end{equation*}
Then, given a function $f\in\mathcal{C}^4(\mathbb{R}^2)$ satisfying the homogeneous Neumann boundary conditions with $f_{i,j} = f(\mathbf{x}_{i,j})$ for $i,j=1,\ldots,M$, it is obtained that
\begin{equation}\label{eq_lap_2D_disct}
    \Delta f(\mathbf{x}_{i,j})\approx \Delta_hf_{i,j} = (L_Mf)_{i,j} + (fL_M^T)_{i,j}.
\end{equation}
For interpolating the variables between primal and dual locations, we define the average matrix
\begin{equation}\label{eq_mat_avg}
    A_M = \frac12
    \begin{bmatrix}
        1 & 1 & 0 & \dots & 0\\
        0 & 1 & 1 & \dots & 0\\
        \vdots & \ddots & \ddots & \ddots & \vdots\\
        0 & \dots & 0 & 1 & 1
    \end{bmatrix}\in \mathbb{R}^{(M-1) \times M}.
\end{equation}
Finally, for matrices $f$, $g$ in $\mathbb{R}^{n\times m}$, we denote their pointwise product by 
$$f\ast g = (f_{i,j}g_{i,j})_{i,j}.$$

\subsubsection{The Operator $\mathcal{C}$}\label{section_operator_C}

The operator $\mathcal{C}$ acts only on the convective subsystem of \eqref{eq_compressible_chns_2D}. 
The fluxes in the $x$- and $y$-directions are, respectively,
\begin{equation*}
F(U) =
\begin{bmatrix}
F^{\rho} \\[5pt]
F^{m_1} \\[5pt]
F^{m_2} \\[5pt]
F^{q}
\end{bmatrix}
=
\begin{bmatrix}
m_1 \\[5pt]
\frac{m^2_1}{\rho} + p(\rho) \\[5pt]
\frac{m_1m_2}{\rho} \\[5pt]
\frac{m_1 q}{\rho}
\end{bmatrix},
\qquad
G(U) =
\begin{bmatrix}
G^{\rho} \\[5pt]
G^{m_1} \\[5pt]
G^{m_2} \\[5pt]
G^{q}
\end{bmatrix}
=
\begin{bmatrix}
m_2 \\[5pt]
\frac{m_1m_2}{\rho} \\[5pt]
\frac{m^2_2}{\rho} + p(\rho) \\[5pt]
\frac{m_2 q}{\rho}
\end{bmatrix}.
\end{equation*}
Let $\hat{F}^{\ast}$ and $\hat{G}^{\ast}$ denote their associated numerical fluxes to $F^{\ast}$ and $G^{\ast}$, respectively.
We compute them using the Rusanov flux with WENO5 reconstructions which are fifth-order accurate for finite differences schemes \cite{BBMZ19a,BBMZ19b,PMP19,Shu09}.
We use this reconstruction since its additional precision does not impact significantly the computational cost.
Let $\mathcal{W}^{x}$ denote the  WENO operator in the horizontal direction.
For a grid function $f(\mathbf{x}_{i,j}) = f_{i,j}$, its right and left reconstructions on primal grids are
\begin{equation*}
  f^{+}_{i+\frac12,j} = \mathcal{W}^{x}\left(f_{i+3,j},\ldots,f_{i-1,j}\right),
  \quad
  f^{-}_{i+\frac12,j} = \mathcal{W}^{x}\left(f_{i-2,j},\ldots,f_{i+2,j}\right),
\end{equation*}
and on dual grids,
\begin{align*}
  f^{+}_{i,j}
    &= \mathcal{W}^{x}\left(f_{i+\frac52,j},\ldots,f_{i-\frac32,j}\right), &
  f^{-}_{i,j}
    &= \mathcal{W}^{x}\left(f_{i-\frac52,j},\ldots,f_{i+\frac32,j}\right), \\[8pt]
  f^{+}_{i+\frac12,j+\frac12}
    &= \mathcal{W}^{x}\left(f_{i+3,j+\frac12},\ldots,f_{i-1,j+\frac12}\right), &
  f^{-}_{i+\frac12,j+\frac12}
    &= \mathcal{W}^{x}\left(f_{i-2,j+\frac12},\ldots,f_{i+2,j+\frac12}\right).
\end{align*}
By the no-slip boundary
conditions we have $m_{1,\frac12,j} = m_{1,M+\frac12,j} = 0$ and $m_{1,i,\frac12} = m_{1,i,M+\frac12} = 0$ for $i,j=1,\ldots,M$. 
Outside the domain, variables are extended by refection: $\rho$ and $q$ symmetrically, and momentum antisymmetrically (see \cite{harlow_welch_65}).

Consequently,
\begin{equation*}
    \begin{split}
        \mathcal{C}(U)_{1,i,j} &\approx -\frac{\hat{F}^{\rho}_{i+\frac12,j} - \hat{F}^{\rho}_{i-\frac12,j}}{h} -\frac{\hat{G}^{\rho}_{i,j+\frac12} - \hat{G}^{\rho}_{i,j-\frac12}}{h},\\
        \mathcal{C}(U)_{2,i+\frac12,j} &\approx -\frac{\hat{F}^{m_1}_{i+1,j} - \hat{F}^{m_1}_{i,j}}{h} -\frac{\hat{G}^{m_1}_{i+\frac12,j+\frac12} - \hat{G}^{m_1}_{i+\frac12,j-\frac12}}{h},\\
        \mathcal{C}(U)_{3,i,j+\frac12} &\approx -\frac{\hat{F}^{m_2}_{i+\frac12,j+\frac12} - \hat{F}^{m_2}_{i-\frac12,j+\frac12}}{h} -\frac{\hat{G}^{m_2}_{i,j+1} - \hat{G}^{m_2}_{i,j}}{h},\\
        \mathcal{C}(U)_{4,i,j} &\approx -\frac{\hat{F}^{q}_{i+\frac12,j} - \hat{F}^{q}_{i-\frac12,j}}{h} -\frac{\hat{G}^{q}_{i,j+\frac12} - \hat{G}^{q}_{i,j-\frac12}}{h},
    \end{split}
\end{equation*}
where the numerical fluxes are given by
\begin{equation*}
    \begin{split}
        &\hat{F}^{\rho}_{i+\frac12,j} 
        =
        \frac{m^{+}_{1,i+1,j} + m^{+}_{1,i,j}}{2}
        - \frac{\lambda^{\rho}_{i+\frac12,j}}{2}\left(\rho^{+}_{i+\frac12,j} - \rho^{-}_{i+\frac12,j}\right),\\[11pt]
        &\hat{F}^{m_1}_{i,j} 
        = \frac12\left(
            \left(\frac{m^2_1}{\rho} + p(\rho)\right)^{+}_{i,j} + \left(\frac{m^2_1}{\rho} + p(\rho)\right)^{-}_{i,j}\right)
            - \frac{\lambda^{m_1}_{i,j}}{2}\left(m^{+}_{1,i,j} - m^{-}_{1,i,j}\right),\\[11pt]
        &\hat{F}^{m_2}_{i+\frac12,j+\frac12} 
        = \frac12\left(
            \left(\frac{m_1m_2}{\rho}\right)^{+}_{i+\frac12,j+\frac12} + \left(\frac{m_1m_2}{\rho}\right)^{-}_{i+\frac12,j+\frac12}\right) 
            \\[8pt]
            &\hspace*{100pt}
            - \frac{\lambda^{m_2}_{i+\frac12,j+\frac12}}{2}\left(m^{+}_{2,i+\frac12,j+\frac12} - m^{-}_{2,i+\frac12,j+\frac12}\right),\\[11pt]
        &\hat{F}^{q}_{i+\frac12,j} 
        = \frac12\left(\left(\frac{m_1 q}{\rho}\right)^{+}_{i+\frac12,j} + \left(\frac{m_1 q}{\rho}\right)^{-}_{i+\frac12,j}\right)
        - \frac{\lambda^{c}_{i+\frac12,j}}{2}\left(q^{+}_{i+\frac12,j} - q^{-}_{i+\frac12,j}\right).
    \end{split}
\end{equation*}
Here, the numerical viscosities $\lambda^\ast$ are defined as the maximum of the upper bounds of the local characteristic speeds at the reconstructed states of each numerical flux,
which are given by $\max\{|v_1| + \sqrt{p^\prime(\rho)}\}$ and $\max\{|v_2| + \sqrt{p^\prime(\rho)}\}$, respectively.

Values of $\rho$ on the dual grid and momentum on the primal grid are obtained using a sixth-order transfer operator
$\mu = \frac{1}{256}\begin{bmatrix}3&-25&150&150&-25&3\end{bmatrix}$.
Hence,
\begin{equation*}
  \rho_{i+\frac12,j} = \sum_{k=i-3}^{i+2}{\mu}_{k+4-i}\rho_{k,j},
  \quad 
  m_{1,i,j} = \sum_{k=i-3}^{i+2}{\mu}_{k+4-i}m_{1,k+\frac12,j},
\end{equation*}
and analogously in the $y$-direction. 
To evaluate $\hat{F}^{m_2}$, $m_1$ is needed at staggered points where it is not directly defined.
So, we first average $m_{1}$ vertically obtaining $m_{1,i+\frac12,j+\frac12}$
and then apply the transfer operator in the $x$-direction to obtain $m_{1,i,j+\frac12}$.
Similarly, for the $\hat{G}^{m_1}$ flux.

\subsubsection{The Operator $\mathcal{L}_1$}
The only nonzero component of the operator $\mathcal{L}_1$ is approximated point-wise, that is,
\begin{equation*}
    \mathcal{L}_1(U)_3 \approx \rho A_M^T g,
\end{equation*}
where the matrix $A_M$ is defined in \eqref{eq_mat_avg}.

\subsubsection{The Operator $\mathcal{L}_2$}
The operator $\mathcal{L}_2$ involves the derivatives of the $c$-variable in the balance of momentum \eqref{eq_compressible_chns_2D_mom_x}-\eqref{eq_compressible_chns_2D_mom_y}.
We employ the following finite-difference approximations, which satisfy the boundary conditions \eqref{bdry_cond},
and are second-order accurate at interior points and first-order accurate otherwise.

Fix $i=1,\ldots,M-1$, $j=1,\ldots,M$. 
The pure double derivative terms in $\mathcal{L}_2(U)_2$ are approximated as
\begin{equation*}
    \begin{split}
        &(c^2_x)_x(\mathbf{x}_{i+\frac12,j}) 
        \approx 
        \begin{cases}
            \displaystyle\frac{\left(c_{M,j} - c_{i,j}\right)^2 - \left(c_{i+1} - c_{1}\right)^2}{4h^3}&\mbox{ if } i=1,\\[11pt]
            \displaystyle\frac{\left(c_{i+2,j} - c_{i,j}\right)^2 - \left(c_{i+1} - c_{i-1}\right)^2}{4h^3}&\mbox{ if } 1<i<M-1,\\[11pt]
            \displaystyle\frac{\left(c_{M,j} - c_{i,j}\right)^2 - \left(c_{i+1} - c_{i-1}\right)^2}{4h^3}&\mbox{ if } i=M-1,
        \end{cases}
        \\[11pt]
        &(c^2_y)_x(\mathbf{x}_{i+\frac12,j}) 
        \approx 
        \begin{cases}
            \displaystyle\frac{\left(c_{i+1,j+1} - c_{i+1,1}\right)^2 - \left(c_{i,j+1} - c_{i,1}\right)^2}{4h^3}&\mbox{ if } j=1,\\[11pt]
            \displaystyle\frac{\left(c_{i+1,j+1} - c_{i+1,j-1}\right)^2 - \left(c_{i,j+1} - c_{i,j-1}\right)^2}{4h^3}&\mbox{ if } 1<j<M,\\[11pt]
            \displaystyle\frac{\left(c_{i+1,M} - c_{i+1,j-1}\right)^2 - \left(c_{i,M} - c_{i,j-1}\right)^2}{4h^3}&\mbox{ if } j=M,
        \end{cases}
    \end{split}
\end{equation*}
For the mixed derivative term $(c_xc_y)_y$, the interior approximation for $j<M$ is
\begin{equation*}
    \begin{split}
        (c_xc_y)_y(\mathbf{x}_{i+\frac12,j})
        &\approx
        \frac{1}{4h^3}
        \left(c_{i+1,j+1} + c_{i+1,j} - c_{i,j+1} - c_{i,j}\right)
        \left(c_{i+1,j+1} + c_{i,j+1} - c_{i+1,j} - c_{i,j}\right)
        \\[11pt]
        &-
        \frac{1}{4h^3}
        \left(c_{i+1,j} + c_{i+1,j-1} - c_{i,j} - c_{i,j-1}\right)
        \left(c_{i+1,j} + c_{i,j} - c_{i+1,j-1} - c_{i,j-1}\right),
    \end{split}
\end{equation*}
and at the boundary points with $j=1,M$ is
\begin{equation*}
    \begin{split}
        (c_xc_y)_y(\mathbf{x}_{i+\frac12,1})
        &\approx
        \frac{1}{4h^3}
        \left(c_{i+1,2} + c_{i+1,1} - c_{i,2} - c_{i,1}\right)
        \left(c_{i,2} + c_{i+1,2} - c_{i,1} - c_{i+1,1}\right)
        ,\\[8pt]
        (c_xc_y)_y(\mathbf{x}_{i+\frac12,M}) 
        &\approx
        \frac{1}{4h^3}
        \left(c_{i+1,M} + c_{i+1,M-1} - c_{i,M} - c_{i,M-1}\right)
        \left(c_{i,M} + c_{i+1,M} - c_{i,M-1} - c_{i+1,M-1}\right)
        .
    \end{split}
\end{equation*}
Analogous expressions are obtained for the $\mathcal{L}_2(U)_3$. 

The nonzero blocks of $\mathcal{L}_2$ expressed in matrix form using the finite-differences matrices \eqref{eq_primal_der}, \eqref{eq_stag_der}, \eqref{eq_mat_avg} are:
\begin{equation*}
    \begin{split}
        \mathcal{L}_2(U)_2
        &\approx
        -\frac{\varepsilon}{2}(D^T_M((c(D^c_M)^T)*(c(D^c_M)^T)) - D^T_M(D^c_Mc*D^c_Mc)) 
        \\[5pt]
        &\qquad\quad
        + \varepsilon((D^T_McA_M^T)*(A_McD_M))D^T_{M},
        \\[8pt]
        \mathcal{L}_2(U)_3
        &\approx
        -\frac{\varepsilon}{2}((D^c_Mc*D^c_Mc)D_M - ((c(D^c_M)^T)*(c(D^c_M)^T))D_M) 
        \\[5pt]
        &\qquad\quad
        + \varepsilon D_M((D^T_McA_M^T)*(A_McD_M)).
    \end{split}
\end{equation*}

\subsubsection{The Operator $\mathcal{L}_3$}\label{section_eyre_splitting}
The treatment of operator $\mathcal{L}_3$, derived from the Cahn-Hilliard-type \eqref{eq_compressible_chns_2D_ch} equation, requires a special care. 
The difficulty stems from the fact that the double-well potential is of a convex-concave type.
As a result, positive definite contributions may appear on the right-hand side making delicate the stability analysis.
A classical strategy, introduced by Eyre in \cite{Eyre98}, ensures unconditional stability for the Cahn-Hilliard equation
by splitting the potential into a convex part $\psi_1$ and a concave part $\psi_2$, 
treating the term with $\psi^\prime_1$ implicitly and the term with $\psi^\prime_2$ explicitly.
Specifically, as in \cite{mulet_24}, we choose
$$\psi^\prime_1(c) = 2c \quad\text{and}\quad \psi^\prime_2(c) = c^3-3c.$$
Consequently, the boundary conditions \eqref{bdry_cond} and
the chain rule yield first-order approximations at points neighboring the boundary and second-order otherwise:
\begin{equation}\label{psi_i_derivatives}
    \begin{split}
        (\Delta\psi^\prime_1(c))(\mathbf{x}_{i,j}) 
        &\approx 2\Delta_h c_{i,j},
        \\[11pt]
        (\psi_2^{\prime\prime}(c)c_x)_{x}(\mathbf{x}_{i,j}), 
        &\approx
        \begin{cases}
            \frac{(\psi_2^{\prime\prime}(c_{i+1,j})+\psi_2^{\prime\prime}(c_{i,j}))(c_{i+1,j}-c_{i,j})}{2h^2}&i=1,\\[3pt]
            \frac{(\psi_2^{\prime\prime}(c_{i+1,j})+\psi_2^{\prime\prime}(c_{i,j}))(c_{i+1,j}-c_{i,j})-(\psi_2^{\prime\prime}(c_{i,j})+\psi_2^{\prime\prime}(c_{i-1,j}))(c_{i,j}-c_{i-1,j})}{2h^2}&1<i<M,\\[3pt]
            \frac{-(\psi_2^{\prime\prime}(c_{i,j})+\psi_2^{\prime\prime}(c_{i-1,j}))(c_{i,j}-c_{i-1,j})}{2h^2}&i=M,\\
        \end{cases}
        \\[11pt]
        (\psi_2^{\prime\prime}(c)c_y)_{y}(\mathbf{x}_{i,j}) 
        &\approx
        \begin{cases}
            \frac{(\psi_2^{\prime\prime}(c_{i,j+1})+\psi_2^{\prime\prime}(c_{i,j}))(c_{i,j+1}-c_{i,j})}{2h^2}&j=1,\\[3pt]
            \frac{(\psi_2^{\prime\prime}(c_{i,j+1})+\psi_2^{\prime\prime}(c_{i,j}))(c_{i,j+1}-c_{i,j})-(\psi_2^{\prime\prime}(c_{i,j})+\psi_2^{\prime\prime}(c_{i,j-1}))(c_{i,j}-c_{i,j-1})}{2h^2}&1<j<M,\\[3pt]
            \frac{-(\psi_2^{\prime\prime}(c_{i,j})+\psi_2^{\prime\prime}(c_{i,j-1}))(c_{i,j}-c_{i,j-1})}{2h^2}&j=M,    
        \end{cases}
    \end{split}
\end{equation}
for $i,j=1,\ldots,M$, and $\Delta_h$ defined in \eqref{eq_lap_2D_disct}. 
Similarly, the derivatives in $y$-direction are obtained.

Notice that by \eqref{bdry_cond},
$$\nabla c\cdot\mathbf{n} = \nabla\left(\frac{\Delta c}{\rho}\right)\cdot\mathbf{n} = 0.$$
Thus, denoting by $D$ the diagonal operator matrix defined as 
$$(D(v) w)_{i,j} = v_{i,j} w_{i,j},$$
for $v,w \in \mathbb{R}^{M\times M}$ we approximate 
$$\Delta\left(\frac{1}{\rho}\Delta c\right)(\mathbf{x}_{i,j}) \approx \left(\Delta_hD(\rho)^{-1}\Delta_h c\right)_{i,j},$$
for $i,j=1,\ldots,M$.

\subsubsection{The Operator $\mathcal{L}_4$}
The operator $\mathcal{L}_4$ accounts for the derivatives of the velocity field appearing in the momentum equation \eqref{eq_compressible_chns_2D_mom_x}-\eqref{eq_compressible_chns_2D_mom_y}. 

Fix $i=1,\ldots,M-1$ and $j=1,\ldots,M$. 
The second-order derivative terms in $\mathcal{L}_4(U)_2$ are approximated as follows:
\begin{equation*}
    \begin{split}
        &\left(v_1\right)_{xx}\left(\mathbf{x}_{i+\frac12, j}\right) =
        \begin{cases}
            \displaystyle\frac{v_{1,i+\frac{3}{2},j} - 2v_{1,i+\frac12,j}}{h^2},&\mbox{for } i=1,\\[11pt]
            \displaystyle\frac{v_{1,i+\frac{3}{2},j} - 2v_{1,i+\frac12,j} + v_{1,i-\frac12,j}}{h^2},&\mbox{for } 1<i<M-1,\\[11pt]
            \displaystyle\frac{v_{1,i-\frac12,j} - 2v_{1,i+\frac12,j}}{h^2},&\mbox{for } i=M-1,
        \end{cases}
        \\[11pt]
        &\left(v_1\right)_{yy}\left(\mathbf{x}_{i+\frac12, j}\right) =
        \begin{cases}
            \displaystyle\frac{v_{1,i+\frac12,j+1} - 3v_{1,i+\frac12,j}}{h^2},&\mbox{for } j=1,\\[11pt]
            \displaystyle\frac{v_{1,i+\frac12,j+1} - 2v_{1,i+\frac12,j} + v_{1,i+\frac12,j-1}}{h^2},&\mbox{for } 1<j<M,\\[11pt]
            \displaystyle\frac{v_{1,i+\frac12,j-1} - 3v_{1,i+\frac12,j}}{h^2},&\mbox{for } j=M,
        \end{cases} 
    \end{split}
\end{equation*}
The cross derivative $(v_2)_{xy}$ is computed by
\begin{equation*}
    \left(v_2\right)_{xy}\left(\mathbf{x}_{i+\frac12, j}\right) =
    \begin{cases}
        \displaystyle\frac{v_{2,i+1, j+\frac12} - v_{2,i, j+\frac12}}{h^2},&\mbox{for } j=1,\\[11pt]
        \displaystyle\frac{(v_{2,i+1, j+\frac12} - v_{2,i, j+\frac12}) - (v_{2,i+1, j-\frac12} - v_{2,i, j-\frac12})}{h^2},&\mbox{for } 1<j<M,\\[11pt]
        \displaystyle\frac{v_{2,i+1, j-\frac12} - v_{2,i, j-\frac12}}{h^2},&\mbox{for } j=M.
    \end{cases}
\end{equation*}
Similar expression are obtained for the block $\mathcal{L}_4(U)_3$.

The three formulas above satisfy the no-slip boundary condition established for velocity field $\mathbf{v}$.
In addition, these approximations are second-order accurate at interior points and first-order accurate near boundaries.

In matrix form, the nonzero blocks of $\mathcal{L}_4(U)$ using \eqref{eq_stag_der} read as
\begin{equation}\label{L_4}
    \begin{split}
    \mathcal{L}_4(U)_2
    &\approx 
    -((2\nu + \lambda)D_M^TD_Mv_1 + \nu v_1(D^{\ast}_{M+1})^TD_{M+1} + (\nu + \lambda)D^T_{M}v_2D^T_M),\\[8pt]    
    \mathcal{L}_4(U)_3
    &\approx 
    -((2\nu + \lambda)v_2D_M^TD_M + \nu D^T_{M+1}D^{\ast}_{M+1}v_2 + (\nu + \lambda)D_{M}v_1D_M).   
    \end{split}
\end{equation}

\subsection{Vector Implementation}

In this section, we rewrite the semi-discrete system \eqref{eq_semi_discrete_system} in vector form. 
Let $\varrho$, $\varrho_{x}$, $\varrho_{y}$, $V_1$, $V_2$, $C$ denote the vectorization of the matrices $\rho$, $(\rho_{i+\frac12,j})_{i,j}$, $(\rho_{i,j+\frac12})_{i,j}$, $v_1$, $v_2$ and $c$, respectively, where the vectorization operator is defined by   
$$\operatorname{vec}(A)_{i+m(j-1)}=A_{i,j},\quad\text{for}\quad 1\leq i\leq n,\enspace 1\leq j\leq m,$$
for any matrix $A\in\mathbb{R}^{n\times m}$. 
Let $I_n$ be the identity matrix of size $n$, and $\otimes$ be the Kronecker product.
We use the symbol $\ast$ to denote the element-wise product between two vector in $\mathbb{R}^n$.
Then,
\begin{equation*}
    \operatorname{vec}(U)= 
    \begin{bmatrix}
        \varrho \\ \varrho_{x} * V_1 \\ \varrho_{y} * V_2 \\ \varrho * C
    \end{bmatrix},
\end{equation*}
and system \eqref{eq_semi_discrete_system} expressed in vector form is
\begin{equation}\label{eq_vec_semi_discrete_system}
    \operatorname{vec}(U^\prime(t)) = \operatorname{vec}(\mathcal{L}(U)).
\end{equation}
From now on, we shall denote $U$, $\mathcal{C}(U)$, $\mathcal{L}_k(U)$ for $k=1,\ldots,4$ both the vectorization and the matrix form whenever there is no risk of confusion.
We only detail the vectorization of the operators explicitly used throughout the paper; the remaining ones are obtained in a similar manner.
\begin{equation*}
  \begin{split}
    \mathcal{L}_3(U)_4 &= 2\Delta_h C + \mathcal{M}_2(C)C - \Delta_h(D(\varrho^{-1})\Delta_h C), \\
    \mathcal{L}_4(U)_{2}&= -A_{1,1}V_1 - A_{1,2}V_2,\\
    \mathcal{L}_4(U)_{3}&= -A_{2,1}V_1 - A_{2,2}V_2,
  \end{split}
\end{equation*}
where $\Delta_h$, $\mathcal{M}_2$ are the tensor form of each $\psi_1$, $\psi_2$ in expressions \eqref{psi_i_derivatives}, respectively, and
\begin{equation}\label{vel_matrices}
    \begin{split}
        A_{1,1} &=(2\nu+\lambda)I_M\otimes (D^T_MD_M)+\nu(D_{M+1}^T D^{\ast}_{M+1}) \otimes I_{M-1},\\
        A_{1,2} &=(\nu+\lambda)D_M \otimes D_M^T,\\
        A_{2,1} &=(\nu+\lambda)D_M^T \otimes D_M,\\
        A_{2,2} &=(2\nu+\lambda)(D_M^T D_M) \otimes I_M + \nu I_{M-1} \otimes (D_{M+1}^T D^{\ast}_{M+1}).\\
    \end{split}
\end{equation}
The one-dimensional case follows immediately as a particular simplification of the above formulation.

\subsection{Implicit-Explicit Schemes}\label{section_imex_schemes}

To solve the semi-discrete ODE system \eqref{eq_vec_semi_discrete_system},
by means of an IMEX partitioned Runge-Kutta scheme \cite{BBMRV15,PR05,mulet_24}, 
we introduce the operator
\begin{equation*}
    \tilde{\mathcal{L}}(\tilde{U}, U) 
    = 
    \mathcal{C}(U) + {\mathcal{L}}_1({U}) + 
    \mathcal{L}_2({U}) + \tilde{\mathcal{L}}_3(\tilde{U}, U) + 
    \mathcal{L}_4(U),
    \quad
    \tilde{U}, U\in\mathbb{R}^N,
\end{equation*}
where the only nonzero block of $\tilde{\mathcal{L}}_3$ is given by
\begin{equation*}
    \begin{split}
        \tilde{\mathcal{L}}_3(\tilde{U}, U)_{4} &= \mathcal{M}_1C + \mathcal{M}_2(\tilde{C})\tilde{C} - \Delta_h(D(\varrho^{-1})\Delta_h C).
    \end{split}
\end{equation*}
By construction, $\mathcal{L}(U) = \tilde{\mathcal{L}}(U, U)$.
Thus, the IVP 
\begin{equation}\label{discrete_ODE}
	\begin{split}
		& U^\prime = \tilde{\mathcal{L}}(U, U),\\
		& U(0) = U_0,
	\end{split}
\end{equation} 
is equivalent to 
\begin{equation}\label{partitioned_RK}
	\begin{split}
		\tilde{U}^\prime &= \tilde{\mathcal{L}}(\tilde{U}, U),\\
		U^\prime &= \tilde{\mathcal{L}}(\tilde{U}, U),\\
		\tilde{U}(0) &= U(0) = U_0.
	\end{split}
\end{equation}
To discretize \eqref{partitioned_RK}, we apply a partitioned Runge-Kutta scheme \cite{mulet_24},
in which the variables with tilde denote the explicit terms, and the remaining ones are treated implicitly.
The method employs two different $s$-stage Butcher tableaus:
\begin{equation*}
	\begin{array}{c|c}
		\tilde{\gamma} & \tilde{\alpha} \\
		\hline
		&\tilde{\beta}^T
	\end{array},
	\qquad
	\begin{array}{c|c}
		\gamma & \alpha \\
		\hline
		&\beta^T
	\end{array},
\end{equation*}
where the first tableau corresponds to the explicit part of the scheme and the second to the diagonally implicit part.
As shown in \cite{PR05}, if both Butcher tableaus are second-order accurate, the resulting partitioned Runge-Kutta method is also second-order accurate.
Assuming that $\beta=\tilde{\beta}$ and $U^n=\tilde{U}^n$, the scheme reads:
\begin{equation}\label{PRK_scheme}
    \begin{split}
        \tilde{U}^{(i)} &= U^n + \Delta t\sum_{j<i}\tilde{\alpha}_{i, j}\tilde{\mathcal{L}}(\tilde{U}^{(j)}, U^{(j)}),\\[2pt]
        U^{(i)} &= U^n + \Delta t\sum_{j<i}\alpha_{i, j}\tilde{\mathcal{L}}(\tilde{U}^{(j)}, U^{(j)}) + \Delta t\alpha_{i, i}\tilde{\mathcal{L}}(\tilde{U}^{(i)}, U^{(i)}),\\[2pt]
        \tilde{U}^{n+1} =        
        U^{n+1} &= U^n + \Delta t\sum_{j=1}^s\beta_{j}\tilde{\mathcal{L}}(\tilde{U}^{(j)}, U^{(j)}),
    \end{split}
  \end{equation}
so there is no need of doubling variables.

\subsection{Solution to the Linear Systems}
At each intermediate stage $i=1,\cdots,s$, the scheme \eqref{PRK_scheme} requires solving a system for $U^{(i)}$ of the form
\begin{equation}\label{nonlinear_eq}
    U^{(i)} = U^n + \Delta t\sum_{j<i}\alpha_{i, j}\tilde{\mathcal{L}}(\tilde{U}^{(j)}, U^{(j)}) + \Delta t\alpha_{i, i}\tilde{\mathcal{L}}(\tilde{U}^{(i)}, U^{(i)}).
\end{equation}
Fixed the $i$ stage, we drop the dependence of the superscript $(i)$ from the variables in $\tilde{U}^{(i)}$ and $U^{(i)}$.
Denoting by hat $\widehat{\cdot}$ the terms computed explicitly in the current stage, namely,
\begin{equation*}
    \widehat{U} = 
    \begin{bmatrix}
        \widehat{\varrho} \\[5pt] \widehat{(\varrho_{x}*V_1)} \\[5pt]
        \widehat{(\varrho_{y}*V_2)} \\[5pt] \widehat{(\varrho\ast C)}
    \end{bmatrix}
    = U^n + \Delta t\sum_{j<i}\alpha_{i, j}\tilde{\mathcal{L}}(\tilde{U}^{(j)}, U^{(j)}) + \Delta t\alpha_{i, i}\mathcal{C}(\tilde{U}),
\end{equation*}
and using \eqref{vel_matrices}, one can rewrite the system above in terms of the density $\varrho$, velocity components $(V_1,V_2)$, and the order parameter $C$:
\begin{subequations}
    \begin{align}
        \varrho - \widehat{\varrho}
        &= \mathbf{0},
        \label{eq_nonlinear_rho}\\[8pt]
        (\varrho_{x}*V_1) - \widehat{(\varrho_{x}*V_1)} + \Delta ta_{i,i}\left(
            A_{1,1}V_1 + A_{1,2}V_2 - \mathcal{L}_2(U)_2\right) 
        &= \mathbf{0}, 
        \label{eq_nonlinear_v1}\\[8pt]
        (\varrho_{y}*V_2) - \widehat{(\varrho_{y}*V_2)} + \Delta ta_{i,i}\left(
            A_{2,1}V_1 + A_{2,2}V_2 - \varrho_{*,y}g - \mathcal{L}_2(U)_3\right) 
        &= \mathbf{0}, 
        \label{eq_nonlinear_v2}\\[8pt]
        (\varrho\ast C) -\widehat{(\varrho\ast C)} + \Delta t \alpha_{i,i}\left(
            \varepsilon\Delta_h\left(\frac{1}{\varrho}\Delta_h C\right) -
                2\Delta_h C - \mathcal{M}_2(\tilde{C})\tilde{C}\right) 
        &= \mathbf{0}.
        \label{eq_nonlinear_c}
    \end{align}
\end{subequations}
where the $\mathcal{L}_2$ operator depends only on $C$. 
Notice that $\varrho$ is explicitly given by \eqref{eq_nonlinear_rho}, specifically,
\begin{equation*}
    \varrho = \varrho^n + \Delta t\sum_{j<i}\!\alpha_{i,j}\mathcal{C}(\tilde{U}^{(j)})_1.
\end{equation*}
Therefore, we can substitute $\varrho$ into \eqref{eq_nonlinear_c}, obtaining a solvable linear system of $M^2$ equations for $C$ taking the form of
\begin{equation}\label{eq_system_c}
    \begin{split}
        &\left(
            D(\varrho) - 2\Delta t\alpha_{i,i}\Delta_h + 
            \Delta t\alpha_{i,i}\varepsilon\Delta_hD(\varrho)^{-1}\Delta_h\right)C 
        = 
        \left(\varrho\ast C\right)^{n} 
        \\[11pt]
        &\qquad
        +\Delta t\sum_{j<i}\!\alpha_{i,j}\mathcal{L}(\tilde{U}^{(j)}, U^{(j)})_4 +
        \Delta t\alpha_{i,i}\left(
            \tilde{\mathcal{C}}(\tilde{U})_4 + \mathcal{M}_2(\tilde{C})\tilde{C}
        \right).
    \end{split}
\end{equation}
Notice that the coefficient matrix in \eqref{eq_system_c} is invertible provided that $\varrho_k>0$ for all $k=1,\cdots,M^2$, since,
due to the convex splitting stated in Section \ref{section_eyre_splitting}, the coefficient matrix is symmetric and positive definite.

Once $\varrho$ and $C$ are already known, we can solve the remaining linear system for the velocity components $V_1$ and $V_2$ obtained from \eqref{eq_nonlinear_v1}-\eqref{eq_nonlinear_v2} of $2M(M-1)$ equations:
\begin{equation}\label{eq_system_v}
    \begin{split}
        &\left(
        \begin{split}
            \begin{bmatrix}
                D(\varrho_{*,x})&\mathbf{0}\\[5pt]
                \mathbf{0}&D(\varrho_{*,y})
            \end{bmatrix}
            + \Delta t\alpha_{i,i}
            \begin{bmatrix}
                A_{1,1} & A_{1,2} \\[5pt]
                A_{2,1} & A_{2,2}
            \end{bmatrix}
        \end{split}
        \right)
        \begin{bmatrix}
            V_1 \\[5pt]
            V_2
        \end{bmatrix}
        =
        \begin{bmatrix}
            {(\varrho_{x}*V_1)^{n}} \\[5pt] {(\varrho_{y}*V_2)^{n}}
        \end{bmatrix}
        \\[11pt]
        &\qquad
        +\Delta t\sum_{j<i}\!\alpha_{i,j}
        \begin{bmatrix}
            \mathcal{L}(\tilde{U}^{(j)}, U^{(j)})_2 \\[5pt]
            \mathcal{L}(\tilde{U}^{(j)}, U^{(j)})_3
        \end{bmatrix}
        +\Delta\alpha_{i,i}
        \begin{bmatrix}
            \mathcal{L}_2(U)_2 \\[5pt]
            \mathcal{L}_1(U)_3 + \mathcal{L}_2(U)_3
        \end{bmatrix}.
    \end{split}
\end{equation}
For analyzing the coefficient matrix in this system we have the following property.
\begin{prop}
    Consider the matrices defined in \eqref{vel_matrices}.
    For $\nu >0, \lambda \geq 0$ the block matrix
    \begin{align*}
    {A} = \begin{bmatrix}  
            {A}_{1,1} & {A}_{1,2} \\ {A}_{2,1} & {A}_{2,2} 
            \end{bmatrix},  
    \end{align*}
    is symmetric and positive definite. 
\end{prop} 
\begin{proof}
  The matrix $A$ can be written as $A = \nu P + (\nu + \lambda) Q$, where 
  \begin{align*}
    &P=\begin{bmatrix}
      I_{M}\otimes ({D}_{M} ^{T}{D}_{M})
             +
         ({D}_{M+1} ^{T}{D}^*_{M+1})\otimes {I}_{M-1},
      &
        \boldsymbol{0}\\[4pt]
        \boldsymbol{0}
      & ({D}_{M} ^{T}D_{M})\otimes {I}_{M}    + 
        {I}_{M-1}\otimes ({D}_{M+1} ^{T}  {D}_{M+1}^*)
     \end{bmatrix},\\[8pt]
    &Q=
    \begin{bmatrix}
      I_{M}\otimes ({D}_{M} ^{T}{D}_{M})
      &
        {D}_{M}\otimes {D}_{M}^T,\\[4pt]
      {D}_{M}^T\otimes {D}_{M}
      & ({D}_{M} ^{T}D_{M})\otimes {I}_{M}   
    \end{bmatrix},\\[8pt]
    &{D}_{M} ^{T}{D}_{M}=
      h^{-2}    \begin{bmatrix}
      2&-1&0&\dots&0\\
      -1&2&-1&\dots&0\\
      \hdotsfor{5}\\
      0&\dots&-1&2&-1\\
      0&\dots&0&-1&2
    \end{bmatrix}                   
    \in\mathbb{R}^{(M-1)\times(M-1)},
    \\[8pt]
    &{D}_{M+1} ^{T}{D}_{M+1}^*=
      h^{-2}\begin{bmatrix}
      2&-1&0&\dots&0\\
      -1&2&-1&\dots&0\\
      \hdotsfor{5}\\
      0&\dots&-1&2&-1\\
      0&\dots&0&-1&2
            \end{bmatrix}          +
      h^{-2}
    \begin{bmatrix}
      1&0&0&\dots&0\\
      0&0&0&\dots&0\\
      \hdotsfor{5}\\
      0&\dots&0&0&0\\
      0&\dots&0&0&1
    \end{bmatrix}                   
    \in\mathbb{R}^{M\times M}.
  \end{align*}
  The matrix $D_{M}^TD_{M}$, the standard Laplacian with Dirichlet conditions on a uniformly spaced grid with  $M-1$ points on $(0,1)$, is symmetric and positive definite and so is 
  ${D}_{M+1}^{T}{D}_{M+1}^*$, being the sum of a symmetric positive definite matrix and a symmetric positive semidefinite (nonnegative diagonal) matrix. 
  Since Kronecker products of symmetric and positive definite matrices are also symmetric and positive definite, we deduce that $P$ is symmetric and positive definite.

  We prove now that $Q$ is (symmetric) positive semidefinite: Let $u\in\mathbb{R}^{(M-1)\times M}$,
  $v\in\mathbb{R}^{M\times (M-1)}$, then
  \begin{align*}
    &\begin{bmatrix}
      \operatorname{vec}(u)^T & \operatorname{vec}(v)^T
    \end{bmatrix}               
    \begin{bmatrix}
      I_{M}\otimes ({D}_{M} ^{T}{D}_{M})
      &
        {D}_{M}\otimes {D}_{M}^T\\[4pt]
      {D}_{M}^T\otimes {D}_{M}
      & ({D}_{M} ^{T}D_{M})\otimes {I}_{M}   
    \end{bmatrix}
    \begin{bmatrix}
      \operatorname{vec}(u)\\[4pt] \operatorname{vec}(v)
    \end{bmatrix}
    \\[8pt]
    &=
    \begin{bmatrix}
      \operatorname{vec}(u)^T & \operatorname{vec}(v)^T
    \end{bmatrix}               
    \begin{bmatrix}
      \operatorname{vec}\Big({D}_{M} ^{T}{D}_{M}u
      +
        {D}_{M}^Tv {D}_{M}^T\Big)\\[4pt]
      \operatorname{vec}\Big({D}_{M}u {D}_{M}
      +v({D}_{M} ^{T}D_{M}\Big)
    \end{bmatrix}
    \\[8pt]
    &=
      \operatorname{vec}(u)^T\operatorname{vec}\Big({D}_{M} ^{T}{D}_{M}u
      +
        {D}_{M}^Tv {D}_{M}^T\Big)+
       \operatorname{vec}(v)^T
      \operatorname{vec}\Big({D}_{M}u {D}_{M}
      +v{D}_{M} ^{T}D_{M}\Big)
    \\[8pt]
    &=
      \text{tr}\Big(u^T\big({D}_{M} ^{T}{D}_{M}u
      +
        {D}_{M}^Tv {D}_{M}^T\big)\Big)+
      \text{tr}\Big(v^T
      \big({D}_{M}u {D}_{M}
      +v{D}_{M} ^{T}D_{M}\big)\Big)
    \\[8pt]
    &=
      \text{tr}(u^T{D}_{M} ^{T}{D}_{M}u)
      +\text{tr}(u^T       {D}_{M}^Tv {D}_{M}^T)+
      \text{tr}(v^T   {D}_{M}u {D}_{M})
      +\text{tr}(v^Tv{D}_{M} ^{T}D_{M})\\[8pt]
    &=
      \text{tr}(u^T{D}_{M} ^{T}{D}_{M}u)
      +\text{tr}(u^T       {D}_{M}^Tv {D}_{M}^T)+
      \text{tr}({D}_{M}v^T   {D}_{M}u )
      +\text{tr}({D}_{M}v^Tv{D}_{M} ^{T})\\[8pt]
    &=\text{tr}\Big((D_Mu+vD_M^T)^T(D_Mu+vD_M^T)\Big) \geq 0,
  \end{align*}
  where $\text{tr}(A)$ denotes the trace of the matrix $A$, which has the properties: 
  $$\operatorname{vec}(A)^T\operatorname{vec}(B)=\text{tr}(A^TB)=\text{tr}(B^TA),$$ 
  for $A, B\in\mathbb{R}^{p\times q}$, 
  $\text{tr}(A)=\text{tr}(A^T)$ for all square matrices and $\text{tr}(A^TA)\geq 0$ for all matrices $A$.

  Therefore, $A$ is symmetric and positive definite.
\end{proof}
The coefficient matrix in \eqref{vel_matrices} is therefore symmetric and positive definite provided that $\varrho_k\geq0$ for every $k=1,\cdots,M^2$.

In \cite{mulet_24}, it was shown that conjugate gradient method and multigrid V-cycle strategy with Gauss-Seidel smoothings serve as effective solvers for systems \eqref{eq_system_c} and \eqref{eq_system_v}, respectively.

\subsection{Time-Step selection}\label{section_time_step}

The time step stability condition is determined by the convective part of system \eqref{eq_compressible_chns_2D}. 
Specifically, for some constant CFL$<1$,
\begin{equation}\label{time_step_cond}
    \Delta t = \text{CFL}\frac{\Delta x}{cs},
\end{equation}
where the maximum characteristic speeds are given by
\begin{equation*}
    cs = \max \left\{
        \left| V^{(i)}_{k,j} \right| + \sqrt{p^\prime(\varrho^{(i)}_j)} : i = 1,\cdots,s,\enspace k = 1, 2,\enspace j=1,\cdots,M^2
    \right\}.
\end{equation*}
In the current work, we do not provide a theoretical proof of the positivity-preserving of $\rho$ and bound-preserving of $c$ of the proposed scheme.
Nevertheless, in our numerical experiments we have observed that if the initial density $\rho_0$ is positive and the initial order parameter $c_0$ lies within the physical bounds $[-1,1]$, then $\rho$ remains positive and $c$ stays close to the physical bounds throughout the simulation.

\section{Numerical experiments}\label{section_numerical_experiments}
In this section, we present several numerical tests to validate that the proposed IMEX scheme achieves second-order accuracy.
We also show that the number of time steps is consistent with the constraints imposed by the purely convective subsystem,
showing that the chosen CFL value guarantees mass conservation and region boundedness of the order parameter $c$.

All the experiments are performed in a two-dimensional setting.
The adiabatic exponent is set to $\frac53$ and the CFL number to $0.4$.
Unless specified otherwise, the remaining parameters are set to
\begin{equation*}
    \nu = 1, \quad \lambda = 0.1, \quad \varepsilon = 10^{-4}, \quad g = -10.
\end{equation*}

All the experiments have been conducted with a MATLAB R2024a implementation on a Linux machine running on 32 core of an AMD EPYC 7282.

\subsection{Order test}
To verify the second-order accuracy of the proposed IMEX scheme, 
we consider a manufactured solution in the domain $\Omega = (0,1)^2$ verifying the boundary conditions \eqref{bdry_cond}.
In particular, the specified solution is given by
\begin{equation*}
    \begin{split}
        \rho(x,y,t) &= 1.25 + 0.1 \cos(2\pi x)\cos(\pi y)(t+1),\\
        v_1(x,y,t) &= \sin(\pi x)\sin(\pi y)(1 - 2t^2),\\
        v_2(x,y,t) &= \sin(\pi x)\sin(2\pi y)(1 + t^2),\\
        c(x,y,t) &= 0.75 + 0.1\cos(\pi x)\cos(\pi y)(1-t).
    \end{split}
\end{equation*}
The global error and order of convergence are computed at time $T = 0.01$ by using the formulas, respectively,
\begin{equation*}
    e_M = \frac{1}{M^2}\sum_{k,i,j}\!|u^{n}_{k,i,j} - u_k(\mathbf{x}_{i,j}, T)|,
    \quad
    \text{EOC}_M = \log_2\left(\frac{e_M}{e_{2M}}\right),
\end{equation*}
for $M=2^{k}$ for $k=3,\ldots,10$.
We choose the following Butcher tableaus 
\begin{equation*}
\begin{aligned}
    &\text{EE--IE:}
    &&
    \begin{array}{c|c}
        0 & 0  \\
        \hline
          & 1 \\
    \end{array}
    &&
    \begin{array}{c|c}
        1 & 1 \\
        \hline
          & 1 \\
    \end{array}
    \\[14pt]
    &{}^\ast\text{DIRKSA:}
    &&
    \begin{array}{c|cc}
        0     & 0     & 0 \\
        1+s   & 1+s   & 0 \\
        \hline
              & s     & 1-s \\
    \end{array}
    &&
    \begin{array}{c|cc}
        1-s & 1-s & 0     \\
        1   &  s  & 1-s   \\
        \hline
            &  s  & 1-s   \\
    \end{array}
    ,\qquad s = \frac{1}{\sqrt{2}}.
\end{aligned}
\end{equation*}
The EE-IE scheme is a first-order IMEX method, while the
$^\ast$-DIRKSA is the only second-order stiffly accurate IMEX scheme with $\alpha_{i,j}\geq0$.
Table \ref{table_orders_convergence_2D} displays the $L_1$ errors and experimental order of convergence for both IMEX schemes.
We also include the errors reported in \cite{mulet_24}, where collocated meshes are used for the same setup. 
Although orders are roughly the same, it can be observed that staggered meshes outperform collocated meshes.
\begin{table}[h]
    \centering
    \resizebox{\textwidth}{!}{
    \begin{tabular}{c c c c c c | c c c c c c}
    \hline
    \multicolumn{6}{c|}{Staggered meshes} & \multicolumn{6}{c}{Collocated grids} \\
    \hline
    \multicolumn{3}{c}{$^\ast$-DIRKSA} & \multicolumn{3}{c|}{EE-IE}
    & \multicolumn{3}{c}{$^\ast$-DIRKSA} & \multicolumn{3}{c}{EE-IE} \\
    \hline
    $M$ & $e_M$ & EOC$_M$ & $M$ & $e_M$ & EOC$_M$
    & $M$ & $e_M$ & EOC$_M$ & $M$ & $e_M$ & EOC$_M$ \\
    \hline
    8   & 1.5423e-02 & 2.02  & 8   & 1.3302e-02 & 2.11 & 8   & 1.9828e-02 & 2.21 & 8   & 1.4989e-02 & 2.25 \\
    16  & 3.8084e-03 & 2.01 & 16  & 3.0819e-03 & 1.77 & 16  & 4.2964e-03 & 2.04 & 16  & 3.1522e-03 & 1.76 \\
    32  & 9.4660e-04 & 2.00 & 32  & 9.0591e-04 & 1.78 & 32  & 1.0422e-03 & 1.97 & 32  & 9.3289e-04 & 1.77 \\
    64  & 2.3639e-04 & 2.00 & 64  & 2.6441e-04 & 1.71 & 64  & 2.6617e-04 & 1.97 & 64  & 2.7460e-04 & 1.73 \\
    128 & 5.9070e-05 & 2.00 & 128 & 8.0562e-05 & 1.58 & 128 & 6.7802e-05 & 1.98 & 128 & 8.2857e-05 & 1.59 \\
    256 & 1.4766e-05 & --- & 256 & 2.6868e-05 & --- & 256 & 1.7148e-05 & ---   & 256 & 2.7457e-05 & ---   \\
    \hline
    \end{tabular}}
    \label{table_orders_convergence_2D}
    \caption{
    $L_1$ errors and experimental order of convergence for the $^\ast$-DIRKSA and EE-IE IMEX schemes
    for the two-dimensional case using a prescribed solution.
    The columns corresponding to staggered meshes refers to the proposed scheme, while the other corresponds to the one described in \cite{mulet_24}.
    }
\end{table}

We henceforth use the $^\ast$-DIRKSA scheme for the remaining numerical tests.

\subsection{Test 1, Test 2 and Test 3}
The following tests aim to show that the $^\ast$-DIRKSA scheme is stable when the order parameter $c$ lies initially both in the unstable region $\left(-\frac{1}{\sqrt{3}},\frac{1}{\sqrt{3}}\right)$ and stable region (see \cite{Elliott89,mulet_24}).
We also show the process of spinodal decomposition.
To this end, mass conservation errors and the minimum and maximum values of $c$ are illustrated.
\subsubsection{Test 1}
We consider the following initial data satisfying boundary conditions
\eqref{bdry_cond}, with $c_0$ taking values in the unstable region
$(-\frac{1}{\sqrt{3}}, \frac{1}{\sqrt{3}})$,
\begin{equation*}
    \begin{split}
        \rho_0(x,y) &= 1.25 + 0.1 \cos(2\pi x)\cos(\pi y),\\
        \mathbf{v}_0(x,y) &= (\sin(\pi x)\sin(\pi y), \sin(\pi x)\sin(2\pi y)),\\
        c_0(x,y) &=  0.1\cos(\pi x)\cos(\pi y).
    \end{split}
\end{equation*}
Initially, at Figure \ref{fig_test1_initial} can be observed that $c_0$ takes values in the unstable region.
\begin{figure}[h!]
    \centering
    \begin{minipage}[b]{0.45\textwidth}
        \centering
        {\small $T=0$} \\[0.3em]
        \includegraphics[width=\textwidth]{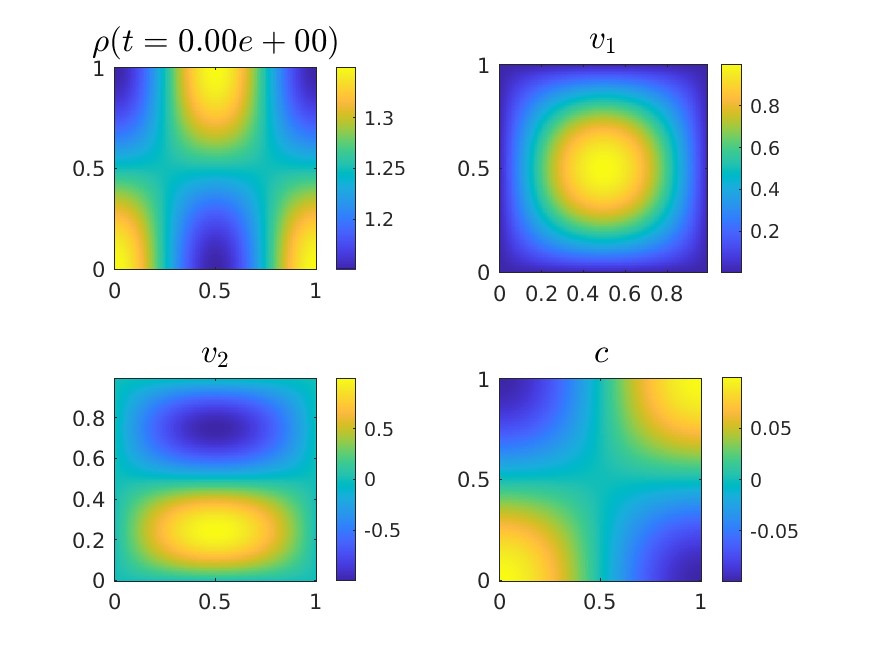}
    \end{minipage}
    \hfill
    \begin{minipage}[b]{0.45\textwidth}
        \centering
        {\small $T=0.1$} \\[0.3em]
        \includegraphics[width=\textwidth]{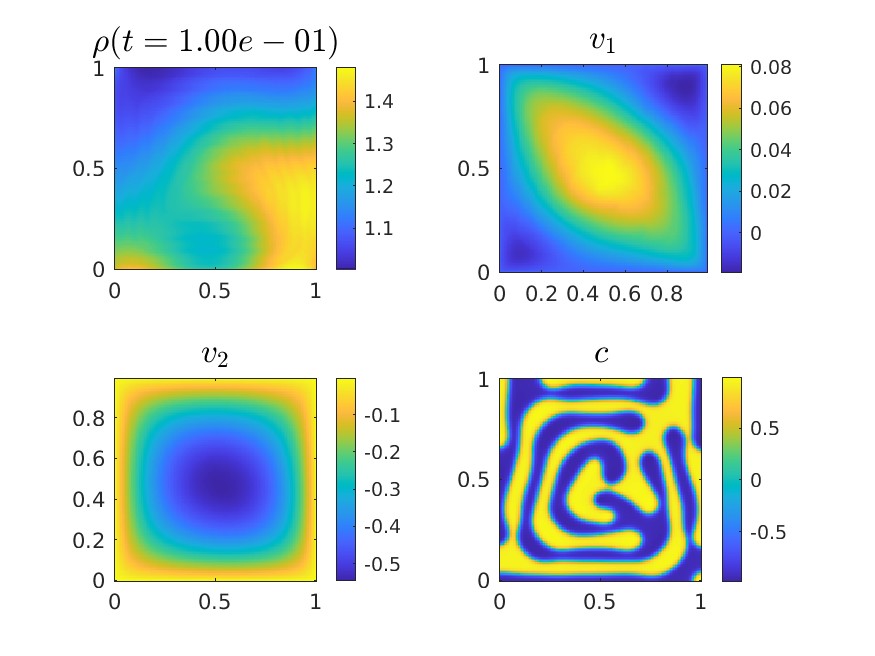}
    \end{minipage}
    \caption{Results for Test 1. 
    Initially, $c$ lies in the unstable region.
    Therefore, phase separation occurs at times $T=0, 0.1$.}
    \label{fig_test1_initial}
\end{figure}
Consequently, at later times, phase separation occurs \cite{CahnHilliard59} forming complex structures, while the density becomes higher at the bottom due to the gravitational effect, as shown in Figures \ref{fig_test1_initial}, \ref{fig_test1_middle}.
\begin{figure}[h!]
    \centering
    \begin{minipage}[b]{0.45\textwidth}
        \centering
        {\small $T=0.3$} \\[0.3em]
        \includegraphics[width=\textwidth]{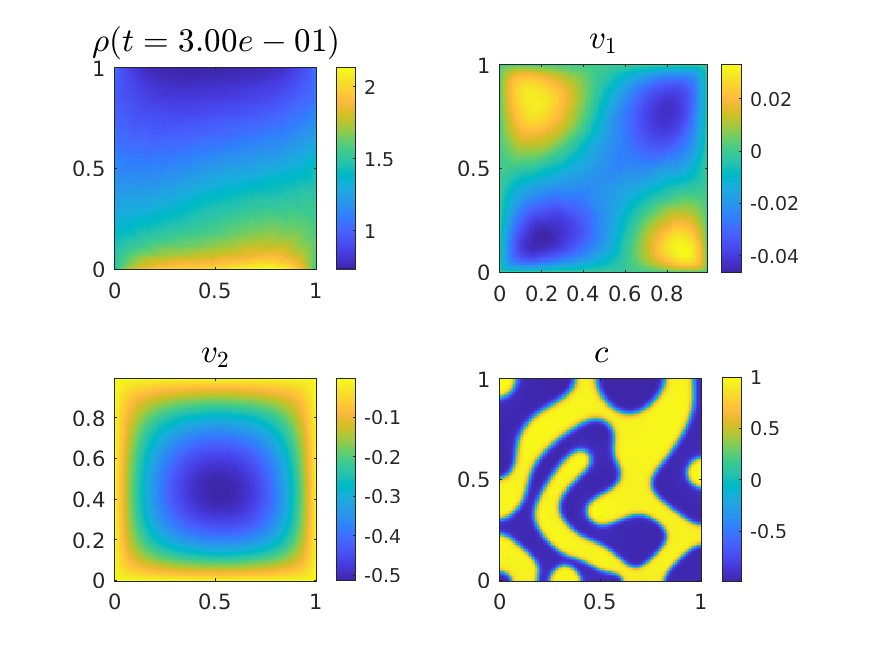}
    \end{minipage}
    \hfill
    \begin{minipage}[b]{0.45\textwidth}
        \centering
        {\small $T=0.5$} \\[0.3em]
        \includegraphics[width=\textwidth]{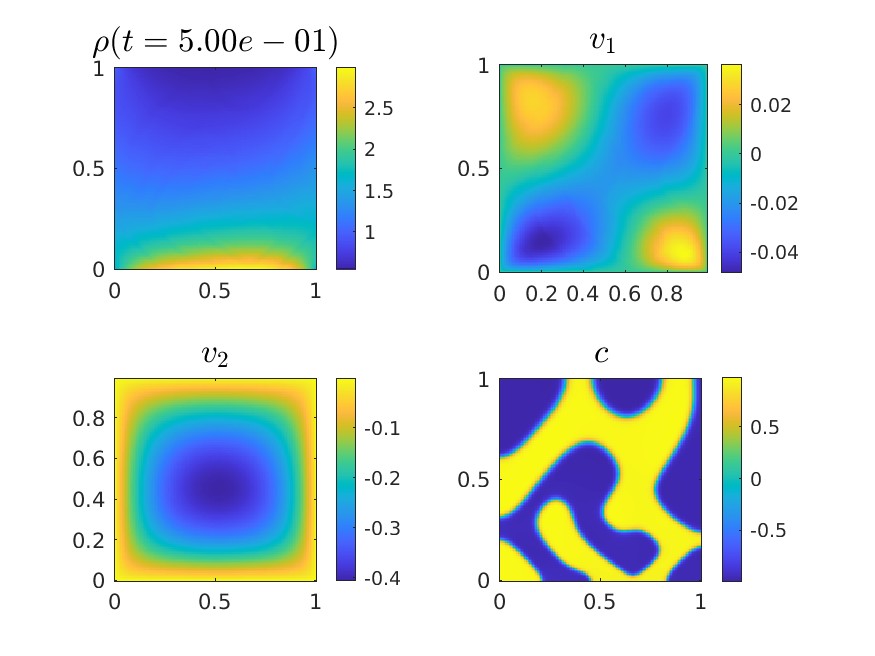}
    \end{minipage}
    \caption{Results for Test 1. 
    The process of phase separation continues, while the density tends to increase at the bottom.}
    \label{fig_test1_middle}
\end{figure}

As time continues evolving, the density tends to increase at the bottom and decrease at the top, 
while the order parameter $c$ tends to separate into two distinct regions, as illustrated in Figure \ref{fig_test1_final}.
Such a process for the order parameter at final times is known as nucleation.
\begin{figure}[h!]
    \centering
    \begin{minipage}[b]{0.45\textwidth}
        \centering
        {\small $T=0.7$} \\[0.3em]
        \includegraphics[width=\textwidth]{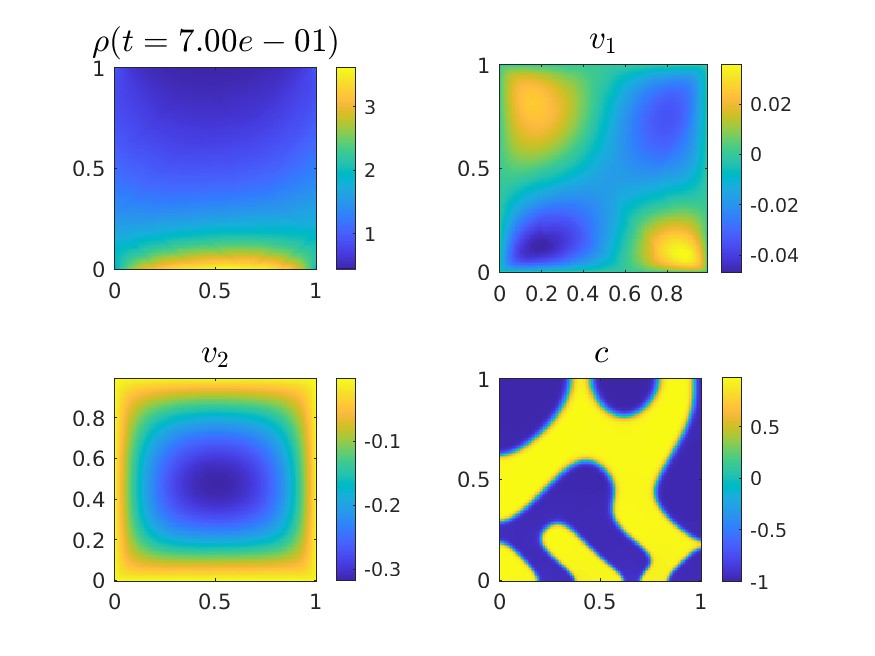}
    \end{minipage}
    \hfill
    \begin{minipage}[b]{0.45\textwidth}
        \centering
        {\small $T=1$} \\[0.3em]
        \includegraphics[width=\textwidth]{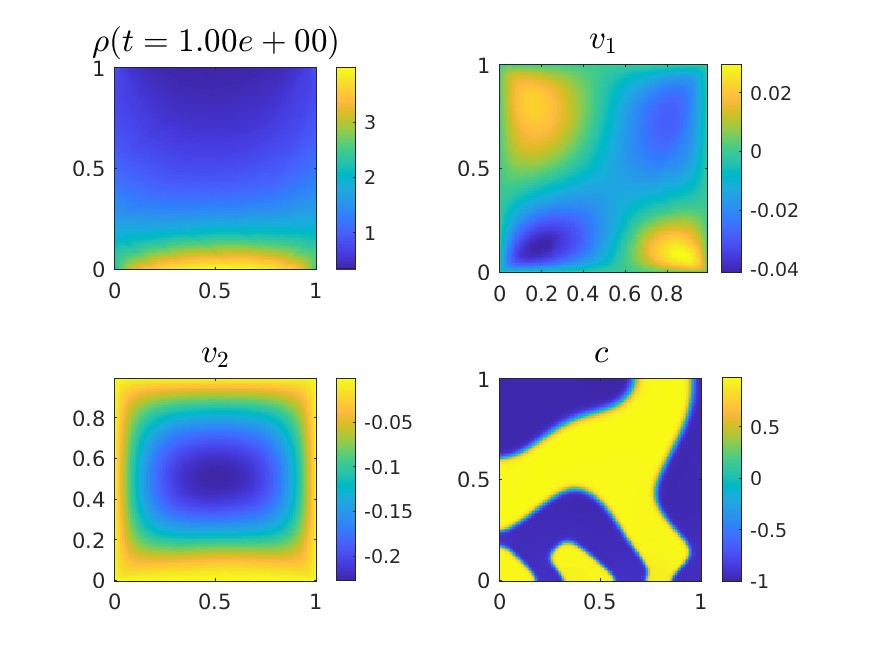}
    \end{minipage}
    \caption{Results for Test 1. 
    The density is higher at the bottom, while the order parameter $c$ tends to separate into two growing regions, leading to nucleation.}
    \label{fig_test1_final}
\end{figure}

\subsubsection{Test 2}
In this test, $c_0$ is in the stable region, and the initial conditions verifying \eqref{bdry_cond} are
\begin{equation*}
    \begin{split}
        \rho_0(x,y) &= 1.25 + 0.1 \cos(2\pi x)\cos(\pi y),\\
        \mathbf{v}_0(x,y) &= (\sin(\pi x)\sin(\pi y), \sin(\pi x)\sin(2\pi y)),\\
        c_0(x,y) &= 0.75 + 0.1\cos(\pi x)\cos(\pi y).
    \end{split}
\end{equation*}
In this case, $c$ lies initially in the stable region, as shown in Figure \ref{fig_test2_initial}.
Consequently, no phase separation occurs, and the order parameter remains close to its initial state, tending to $\frac34$ (see Figures \ref{fig_test2_middle} and \ref{fig_test2_final}). 
Therefore, the fluid is governed by the compressible Navier-Stokes equations with gravitational forces, as illustrated in Figures \ref{fig_test2_initial}, \ref{fig_test2_middle}, and \ref{fig_test2_final}.
\begin{figure}[h!]
    \centering
    \begin{minipage}[b]{0.45\textwidth}
        \centering
        {\small $T=0$} \\[0.3em]
        \includegraphics[width=\textwidth]{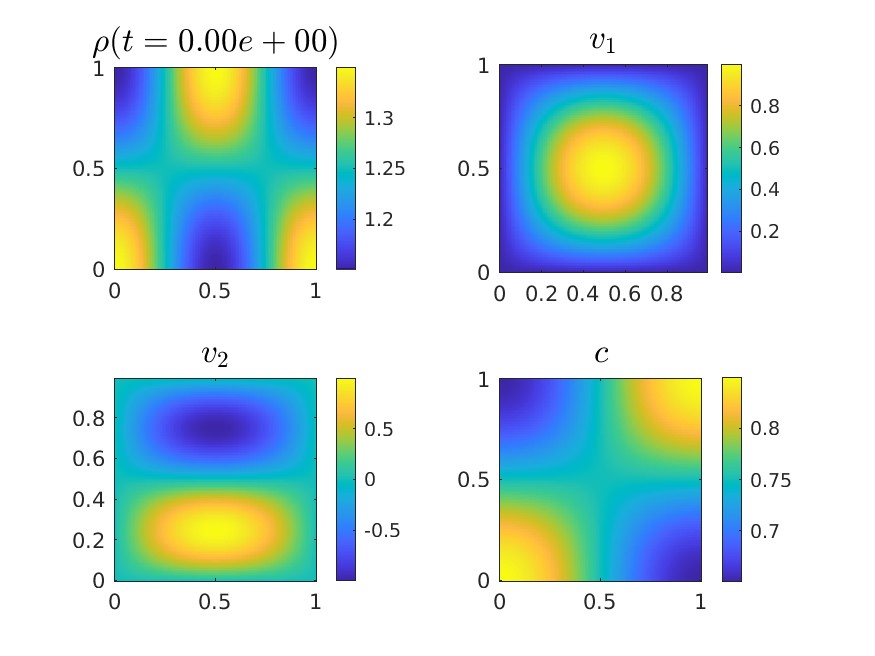}
    \end{minipage}
    \hfill
    \begin{minipage}[b]{0.45\textwidth}
        \centering
        {\small $T=0.1$} \\[0.3em]
        \includegraphics[width=\textwidth]{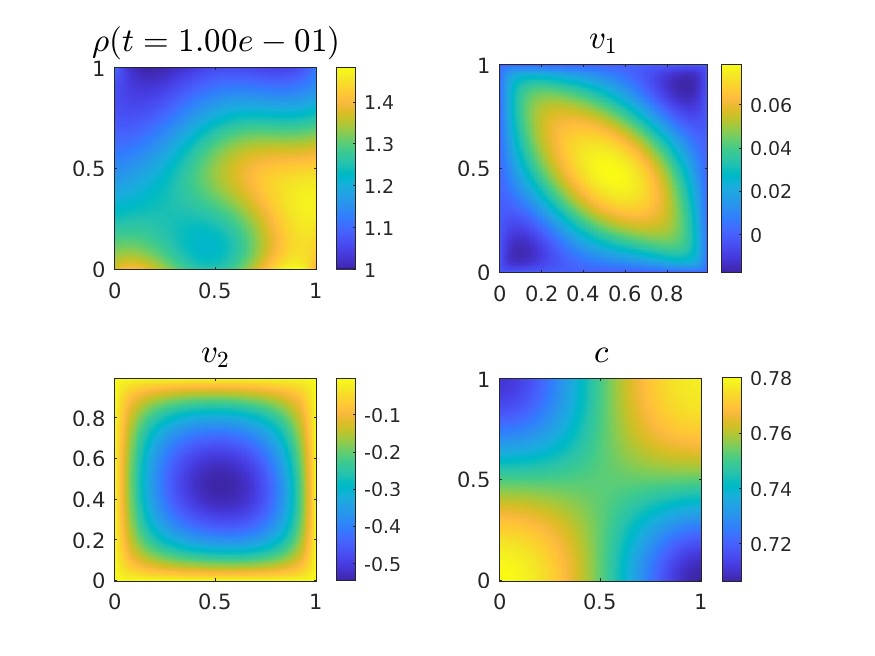}
    \end{minipage}
    \caption{Results for Test 2. Initially, $c$ lies in the stable region starting to tend to $\frac34$.}
    \label{fig_test2_initial}
\end{figure}
\begin{figure}[h!]
    \centering
    \begin{minipage}[b]{0.45\textwidth}
        \centering
        {\small $T=0.3$} \\[0.3em]
        \includegraphics[width=\textwidth]{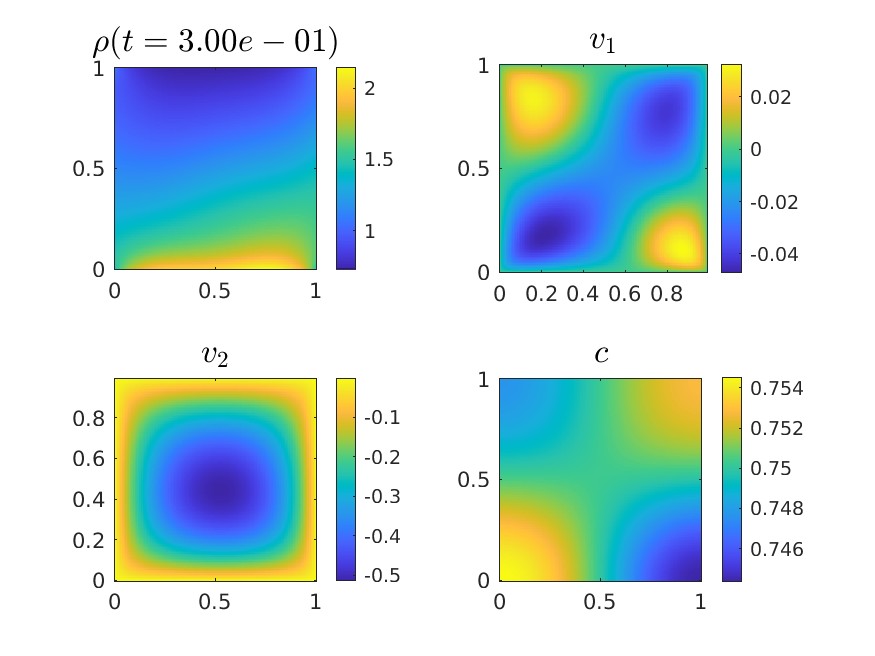}
    \end{minipage}
    \hfill
    \begin{minipage}[b]{0.45\textwidth}
        \centering
        {\small $T=0.5$} \\[0.3em]
        \includegraphics[width=\textwidth]{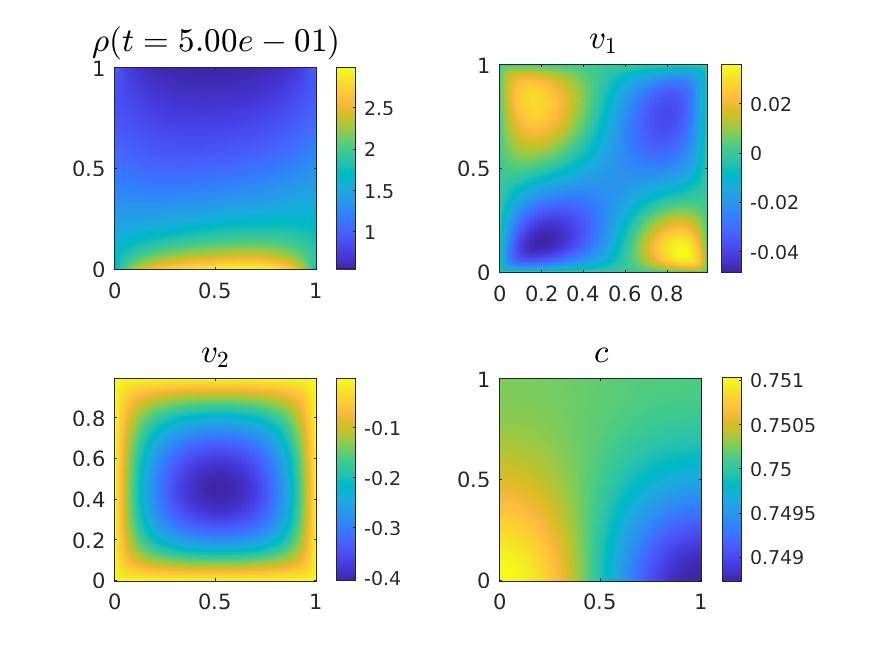}
    \end{minipage}
    \caption{Results for Test 2. The order parameter $c$ remains in the stable region, and at $T=0.5$ it is close to $\frac34$, while the density starts increasing at the bottom.}
    \label{fig_test2_middle}
\end{figure}

\begin{figure}[h!]
    \centering
    \begin{minipage}[b]{0.45\textwidth}
        \centering
        {\small $T=0.7$} \\[0.3em]
        \includegraphics[width=\textwidth]{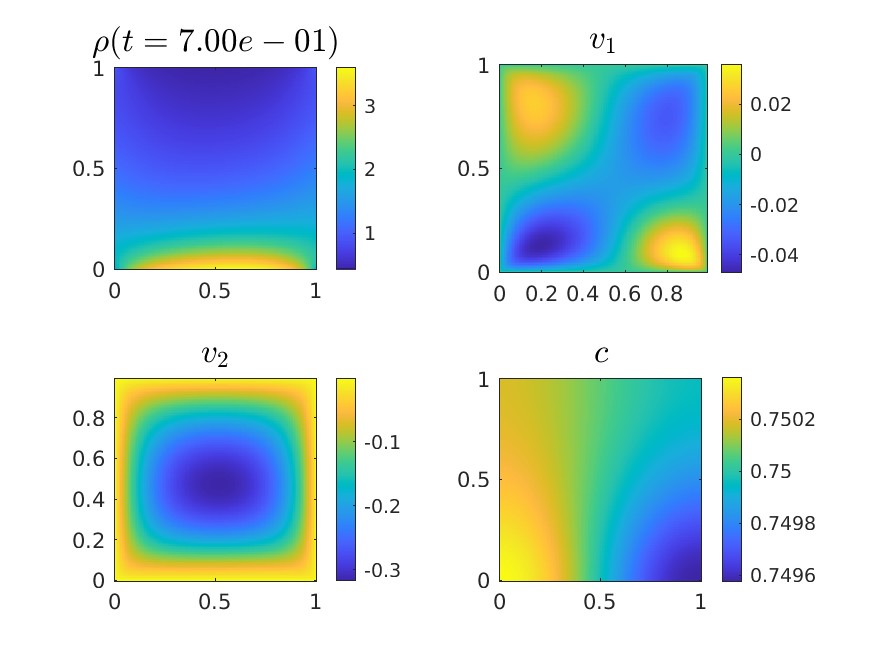}
    \end{minipage}
    \hfill
    \begin{minipage}[b]{0.45\textwidth}
        \centering
        {\small $T=1$} \\[0.3em]
        \includegraphics[width=\textwidth]{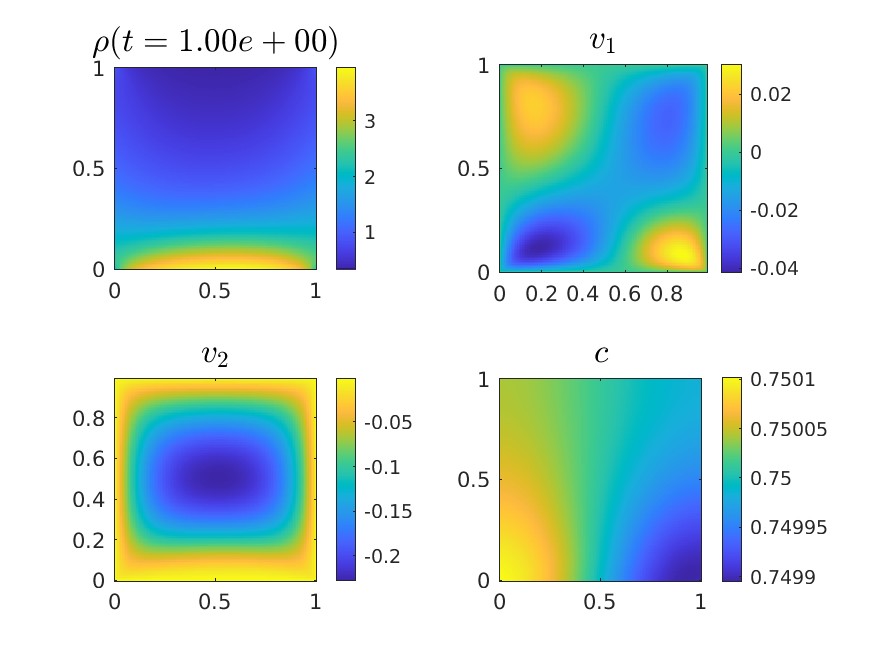}
    \end{minipage}
    \caption{Results for Test 2. The $c$ variable becomes closer to $\frac34$ pointing out that the fluid is governed by the isentropic compressible Navier-Stokes equations.}
    \label{fig_test2_final}
\end{figure}

\subsubsection{Test 3}
This test illustrates the process of spinodal decomposition.
We set $\rho_0=1$, $\mathbf{v}_0 = 0$ and initialize $c_0$ as a normally distributed random sample with mean $0$ and variance $10^{-10}$. 
The viscosity coefficients are chosen as $\nu=10^{-3}$, $\lambda = 10^{-4}$.
\begin{figure}[h]
    \centering
    \begin{minipage}[b]{0.45\textwidth}
        \centering
        {\small $T=0$} \\[0.3em]
        \includegraphics[width=\textwidth]{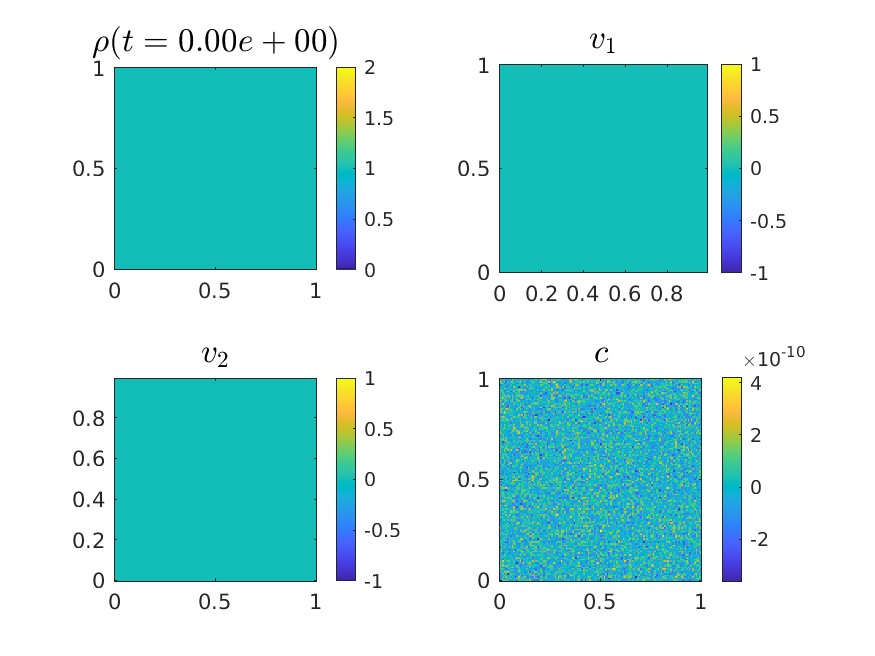}
    \end{minipage}
    \hfill
    \begin{minipage}[b]{0.45\textwidth}
        \centering
        {\small $T=0.1$} \\[0.3em]
        \includegraphics[width=\textwidth]{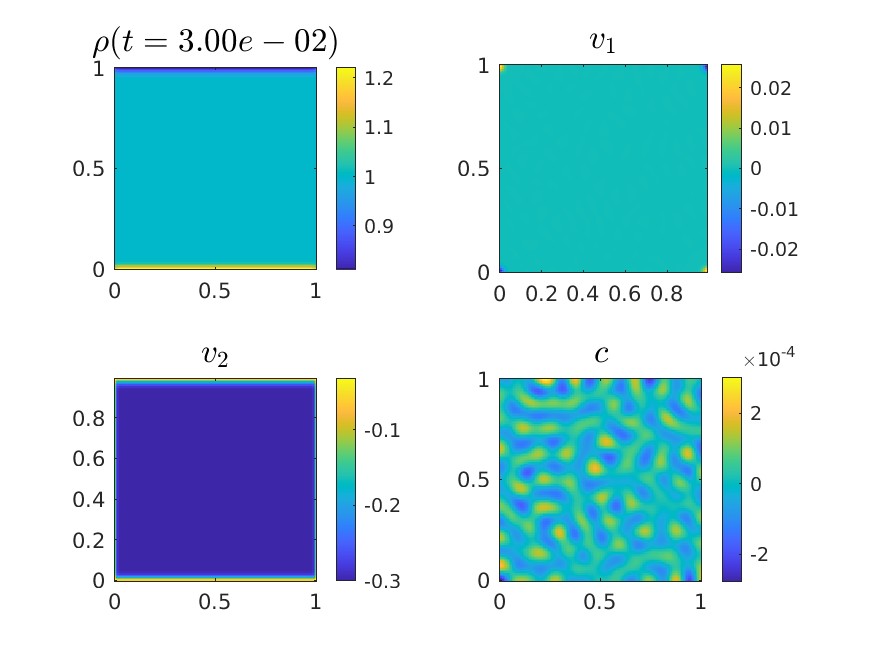}
    \end{minipage}
    \caption{Results for Test 3. Initially, $c$ lies in the unstable region.
    Phase separation starts to occur, for example, at time $T=0.1$.}
    \label{fig_test3_initial}
\end{figure}
For this test, the initial value of $c$ clearly lies in the unstable region, triggering phase separation in the early stages of the simulation and eventually leading to a nucleation regime.
\begin{figure}[h]
    \begin{minipage}[b]{0.45\textwidth}
        \centering
        {\small $T=0.3$} \\[0.3em]
        \includegraphics[width=\textwidth]{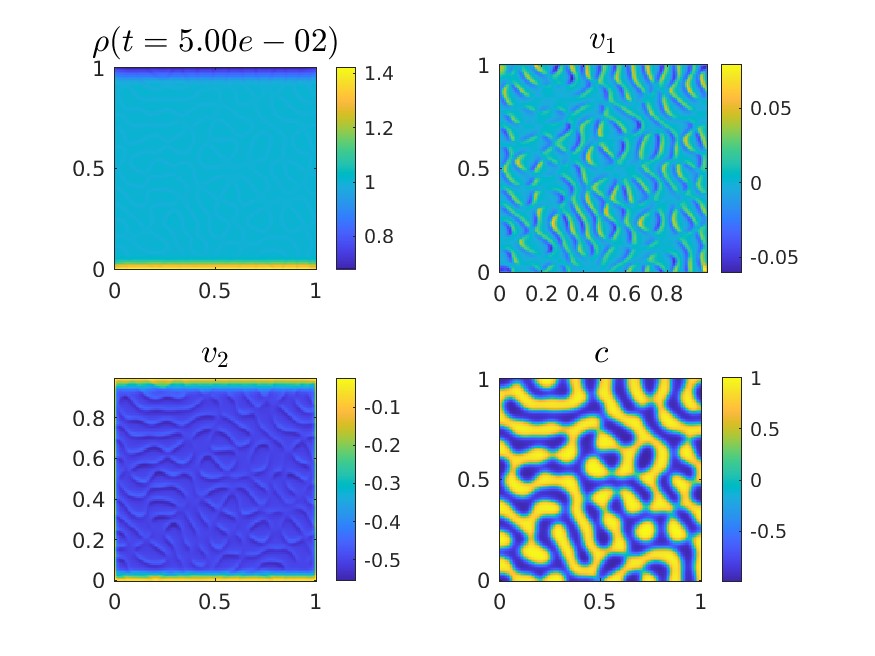}
    \end{minipage}
    \hfill
    \begin{minipage}[b]{0.45\textwidth}
        \centering
        {\small $T=0.5$} \\[0.3em]
        \includegraphics[width=\textwidth]{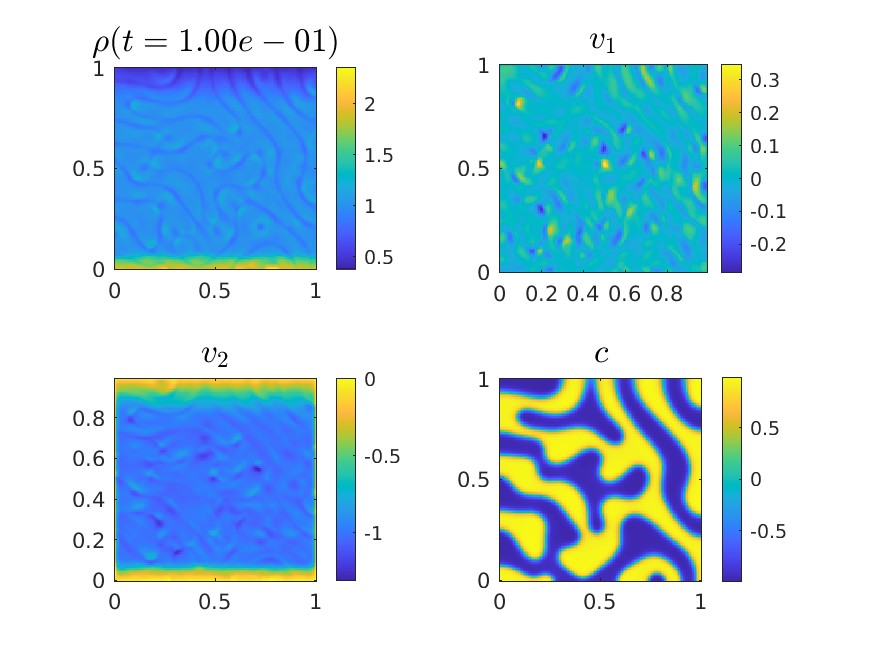}
    \end{minipage}
    \caption{Results for Test 3. 
    Phase separation occurs at times $T=0.3, 0.5$.
    The density becomes higher at the bottom due to gravitational effects.}
    \label{fig_test3_middle}
\end{figure}
\begin{figure}[h]
    \centering
    \begin{minipage}[b]{0.45\textwidth}
        \centering
        {\small $T=0.7$} \\[0.3em]
        \includegraphics[width=\textwidth]{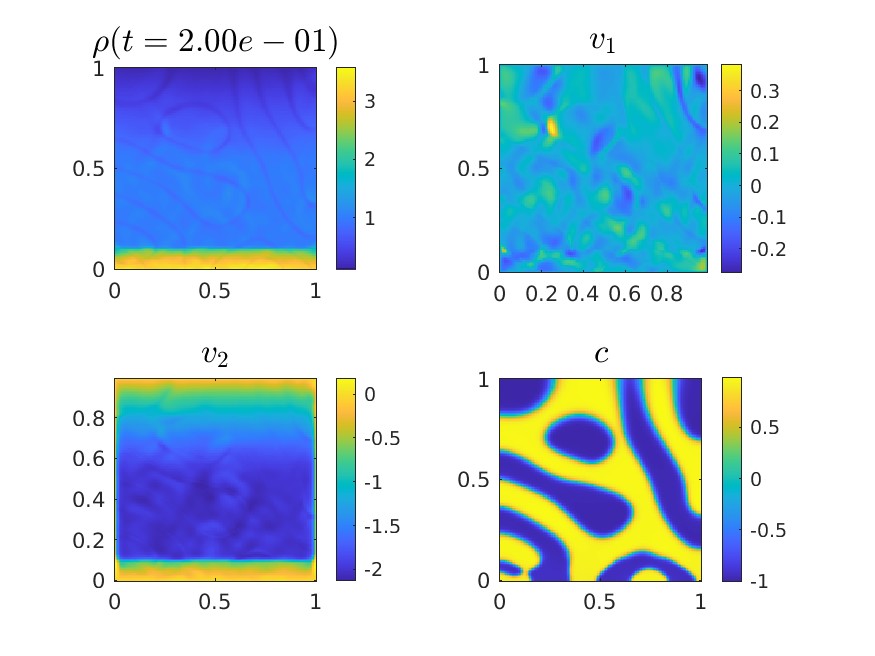}
    \end{minipage}
    \hfill
    \begin{minipage}[b]{0.45\textwidth}
        \centering
        {\small $T=1$} \\[0.3em]
        \includegraphics[width=\textwidth]{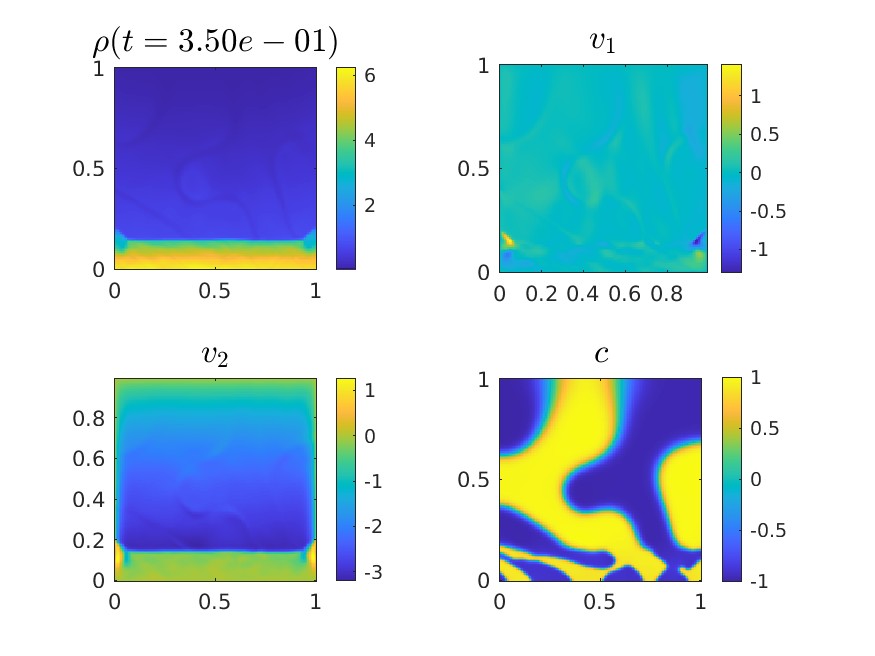}
    \end{minipage}
    \caption{Results for Test 3. Nucleation process is observed at $T=0.7, 1$. 
    The density becomes higher at the bottom due to gravitational effects.}
    \label{fig_test3_end}
\end{figure}
As the dynamics evolves, the density becomes higher near the bottom of the domain. 
The results are shown in Figure \ref{fig_test3_initial}, \ref{fig_test3_middle}, and \ref{fig_test3_end}.

\subsubsection{Mass conservation and preservation of bounds}
Figures \ref{fig_c_evol_tests} and \ref{fig_mass_conservation_tests} illustrates the mass conservation error and the minimum and maximum values of $c$ for Tests 1, 2, and 3, computed by
\begin{equation*}
    \begin{split}
        \text{err}_{\rho}(t^n) = {\sum_{i,j}(\rho^n_{i,j} - \rho^0_{i,j})},
        \quad
        \text{err}_{q}(t^n) = {\sum_{i,j}((\rho c)^n_{i,j} - (\rho c)^0_{i,j})}.
    \end{split}
\end{equation*}
It is observed that the order parameter $c$ remains roughly within the physical bounds $[-1,1]$ throughout the three simulations.
Moreover, the mass conservation error for both $\rho$ and $q=\rho c$ is preserved.
Therefore, the choice of the CFL parameter is adequate to ensure both stability and bound preservation.

\begin{figure}[h!]
  \centering
  \begin{minipage}[b]{0.32\textwidth}
        \centering
        \includegraphics[width=\textwidth]{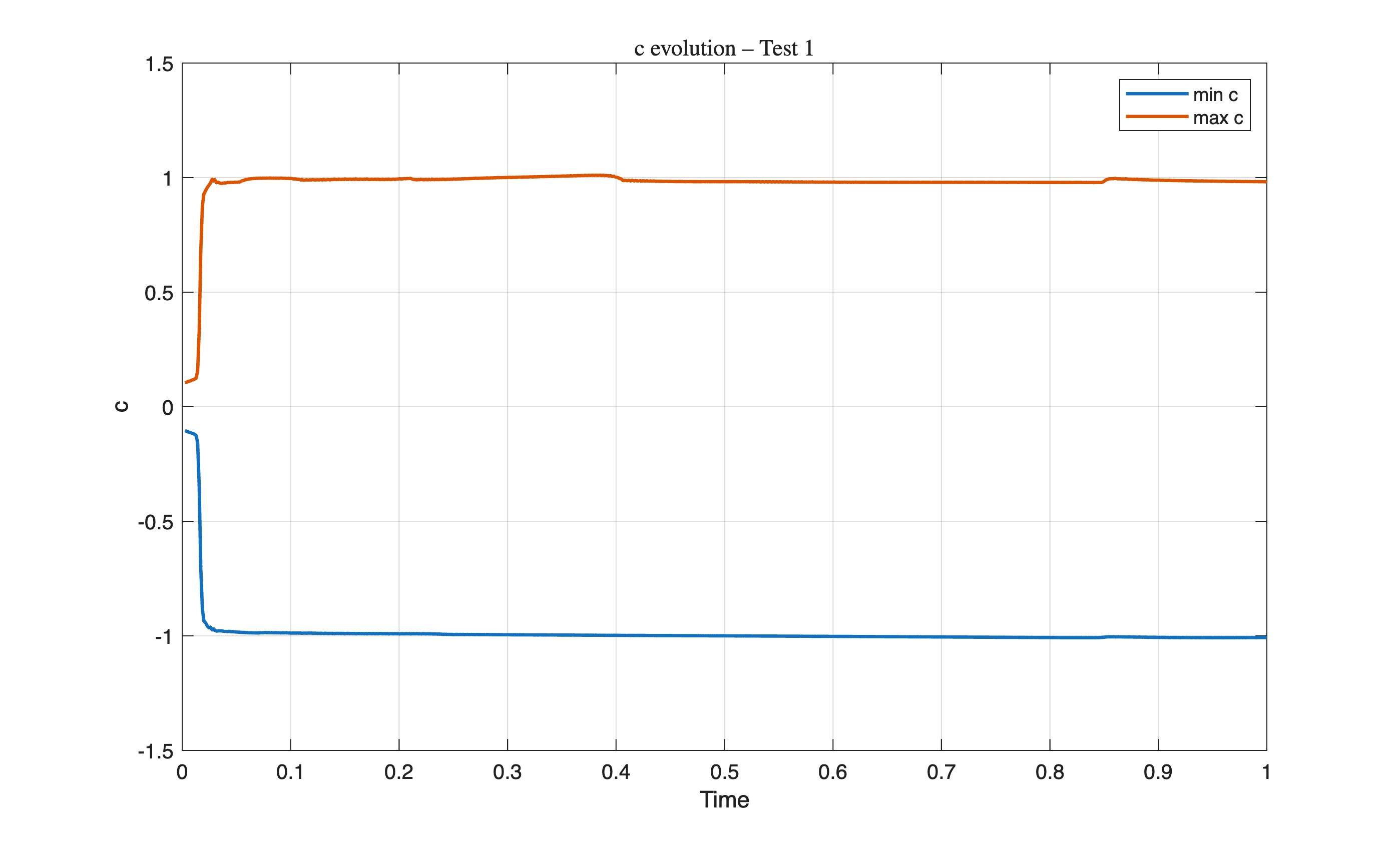}
    \end{minipage}
    \hfill
    \begin{minipage}[b]{0.32\textwidth}
        \centering
        \includegraphics[width=\textwidth]{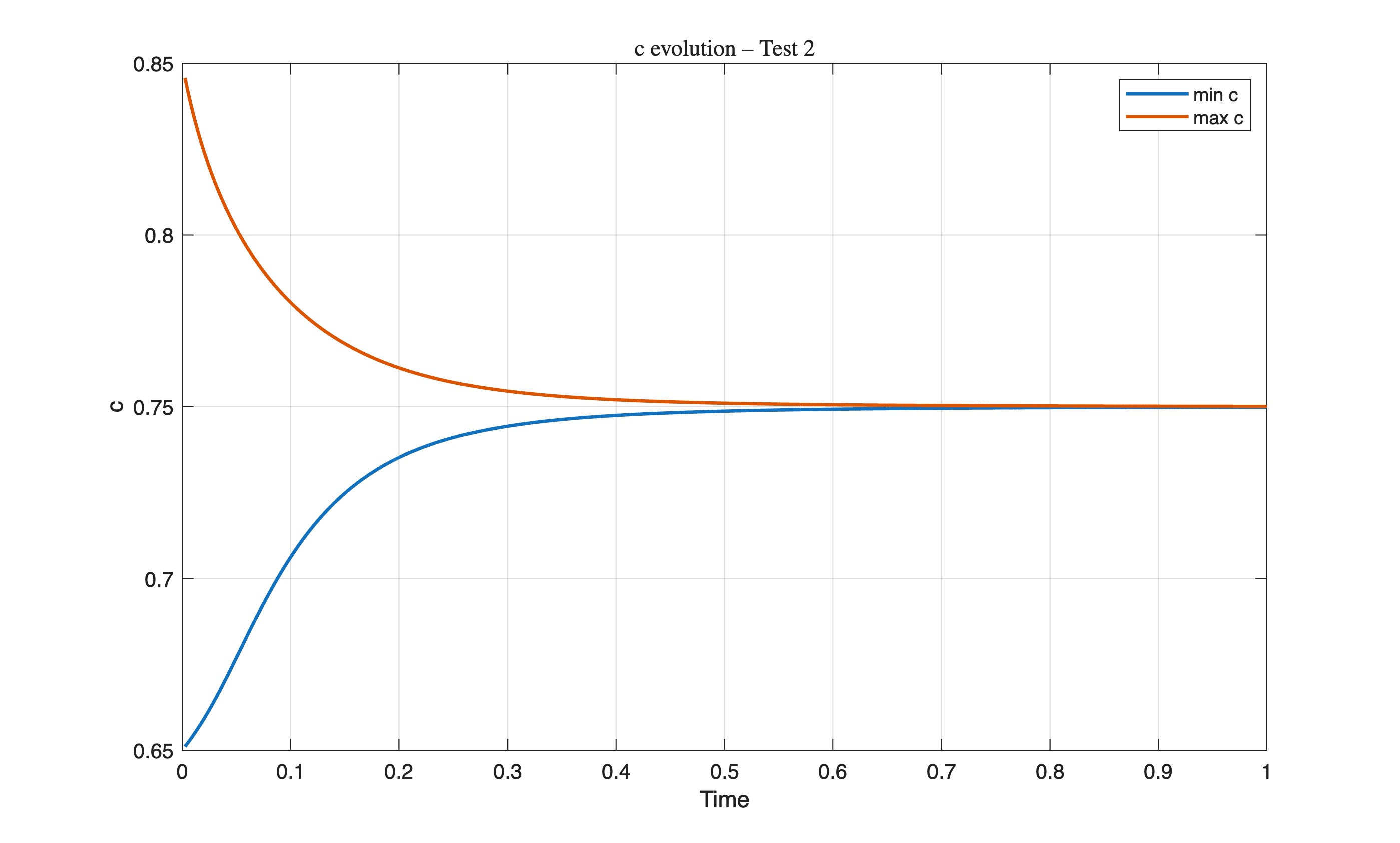}
    \end{minipage}
    \hfill
    \begin{minipage}[b]{0.32\textwidth}
        \centering
        \includegraphics[width=\textwidth]{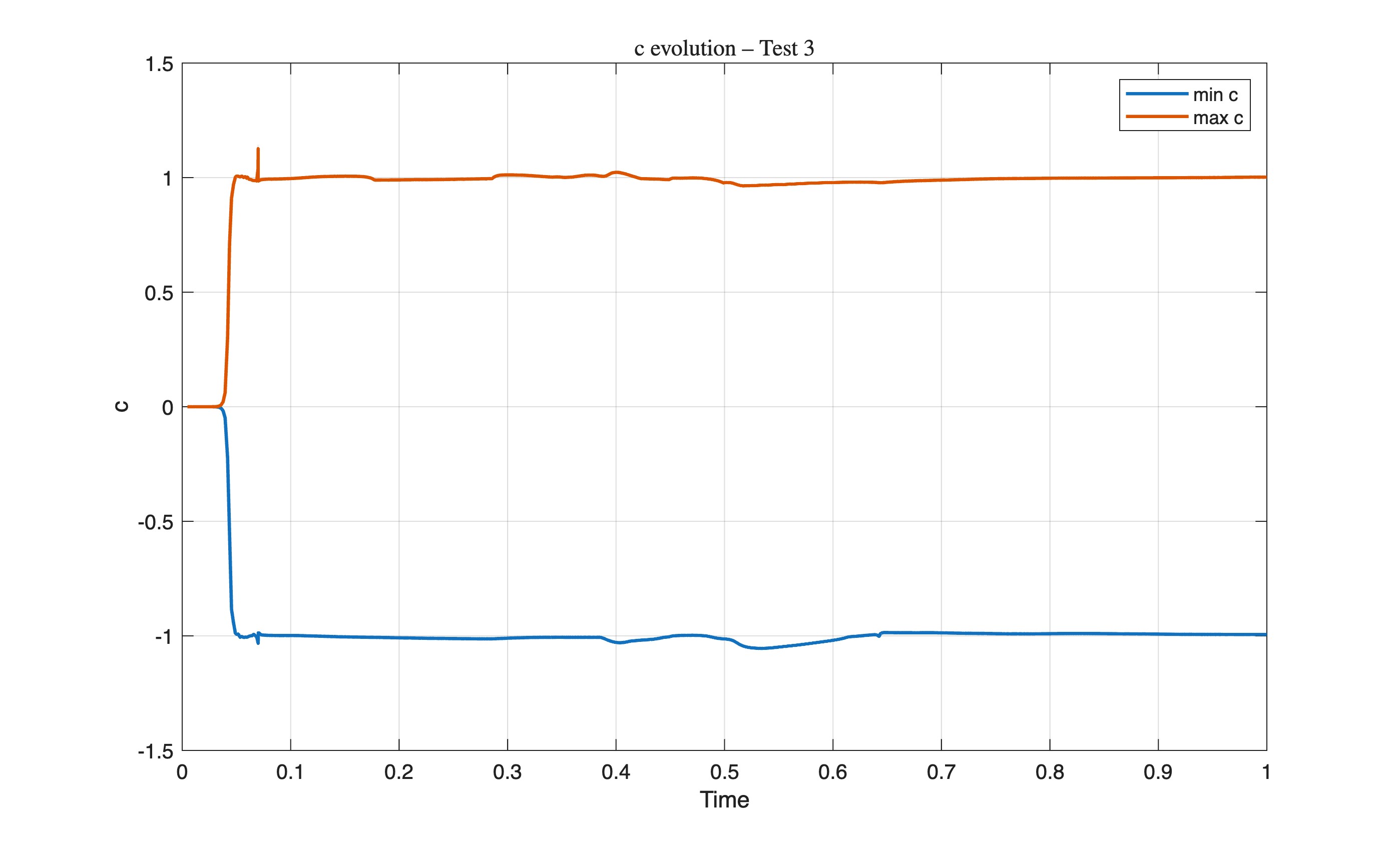}
    \end{minipage}
    \caption{Time evolution of the maximum and minim value of the order parameter $c$ for Test 1 (left), Test 2 (middle) and Test 3 (right).}
    \label{fig_c_evol_tests}
\end{figure}
\begin{figure}[h!]
  \centering
  \begin{minipage}[b]{0.32\textwidth}
        \centering
        \includegraphics[width=\textwidth]{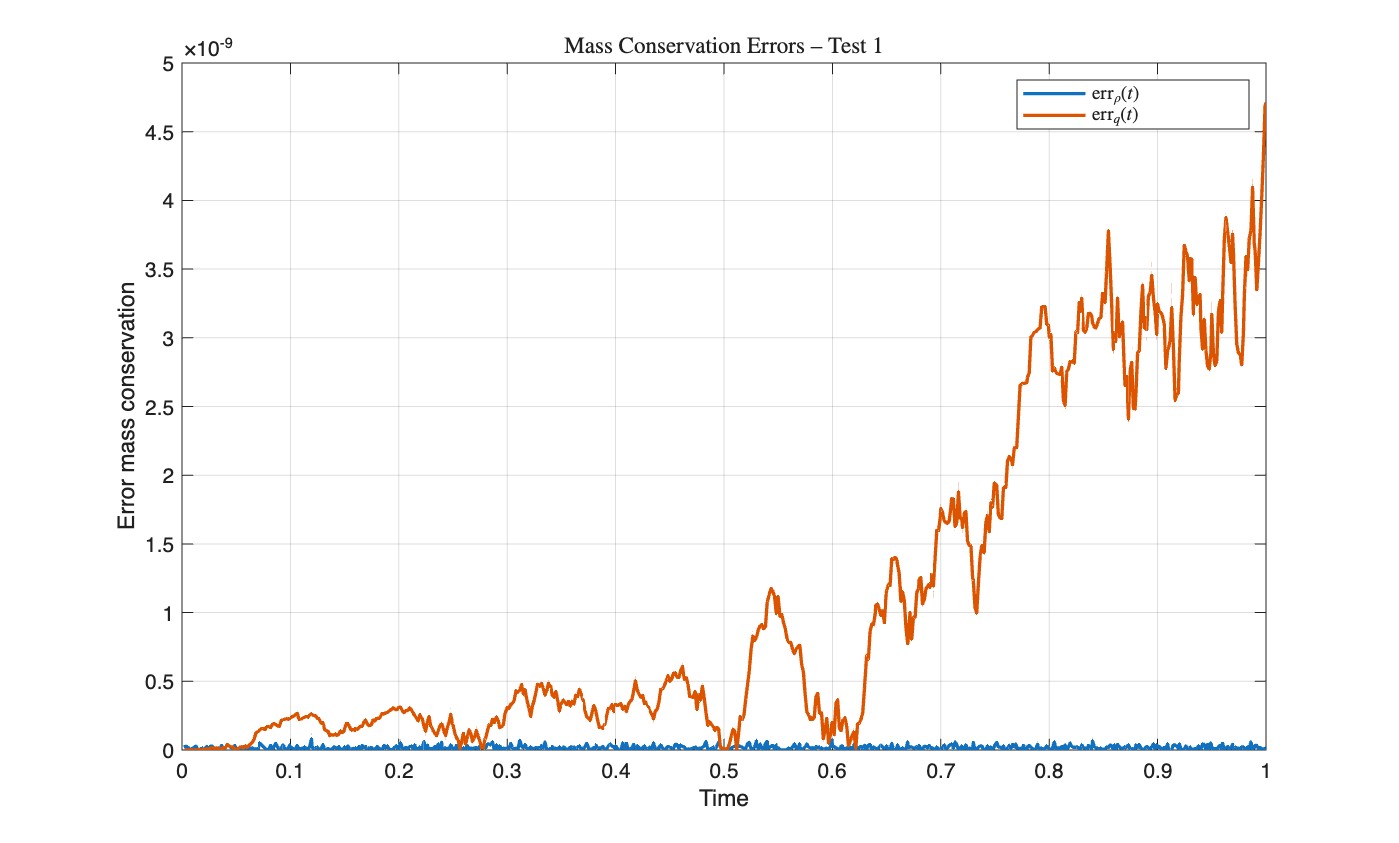}
    \end{minipage}
    \hfill
    \begin{minipage}[b]{0.32\textwidth}
        \centering
        \includegraphics[width=\textwidth]{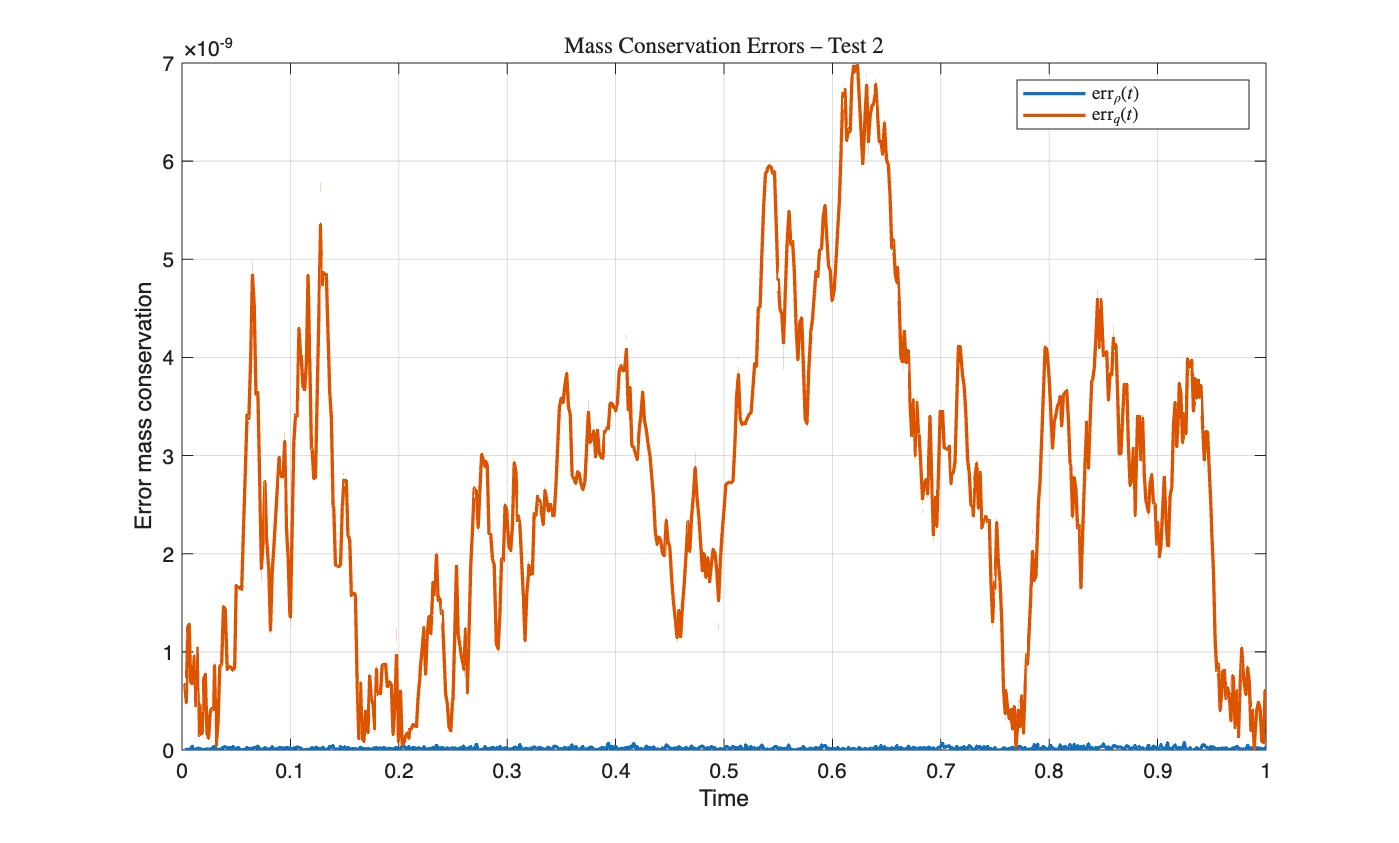}
    \end{minipage}
    \hfill
    \begin{minipage}[b]{0.32\textwidth}
        \centering
        \includegraphics[width=\textwidth]{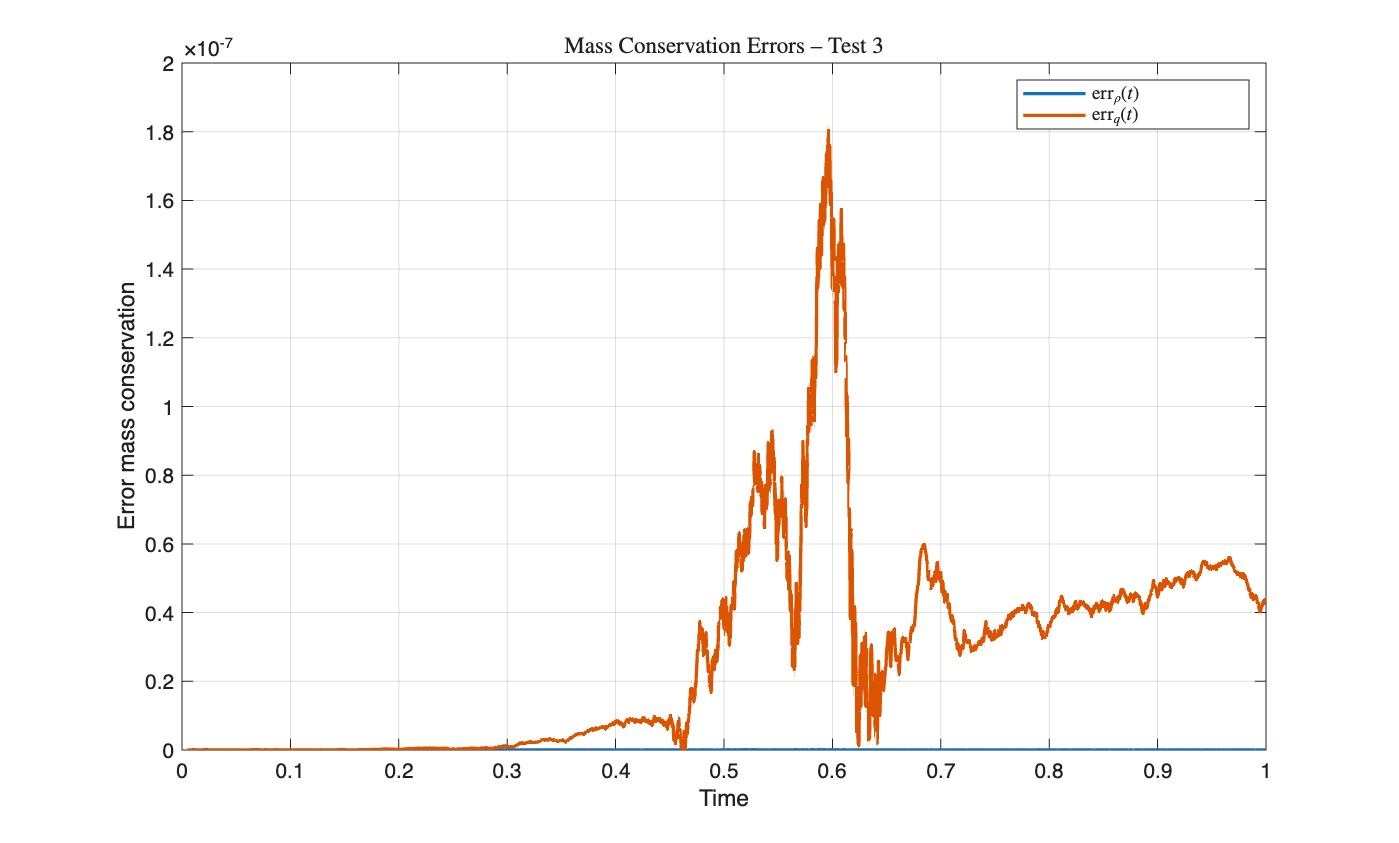}
    \end{minipage}
    \caption{Time evolution of mass conservation errors for $\rho$ and $q=\rho c$ for Test 1 (left), Test 2 (middle) and Test 3 (right).}
    \label{fig_mass_conservation_tests}
\end{figure}

\subsection{Test 4}
The following test, taken from \cite{ChenZhao20,mulet_24}, aims to show the dynamics of two kissing bubbles in a compressible fluid.
The initial condition satisfying boundary conditions \eqref{bdry_cond} is given by
\begin{equation*}
    \begin{aligned}
    \rho_0(x,y) &= 1, \\
    \mathbf{v}_0(x,y) &= (0,0), \\
    c_0(x,y) &=
    \tanh\!\left(
    \frac{
    \sqrt{(x - 0.4)^2 + (y - 0.5)^2} - 0.1
    }{2\varepsilon\sqrt{2}}
    \right) \\
    &\quad \cdot
    \tanh\!\left(
    \frac{
    \sqrt{(x - 0.6)^2 + (y - 0.5)^2} - 0.1
    }{2\varepsilon\sqrt{2}}
    \right).
    \end{aligned}
\end{equation*}
For this test the pressure coefficient is set to $C_p=10^4$, the viscosity coefficients to $\nu=\lambda=0.1$ and $\varepsilon=0.01$.
The merging process of the two bubbles is illustrated in Figure \ref{fig_test4_0level_c}, where the $0$-level sets of the order parameter $c$ are plot. 
Initially, the two bubbles are tangent at the point $(0.5, 0.5)$.
As time evolves, the two bubbles merge into a single bubble.
\begin{figure}[h]
    \centering
    \includegraphics[width=0.5\textwidth]{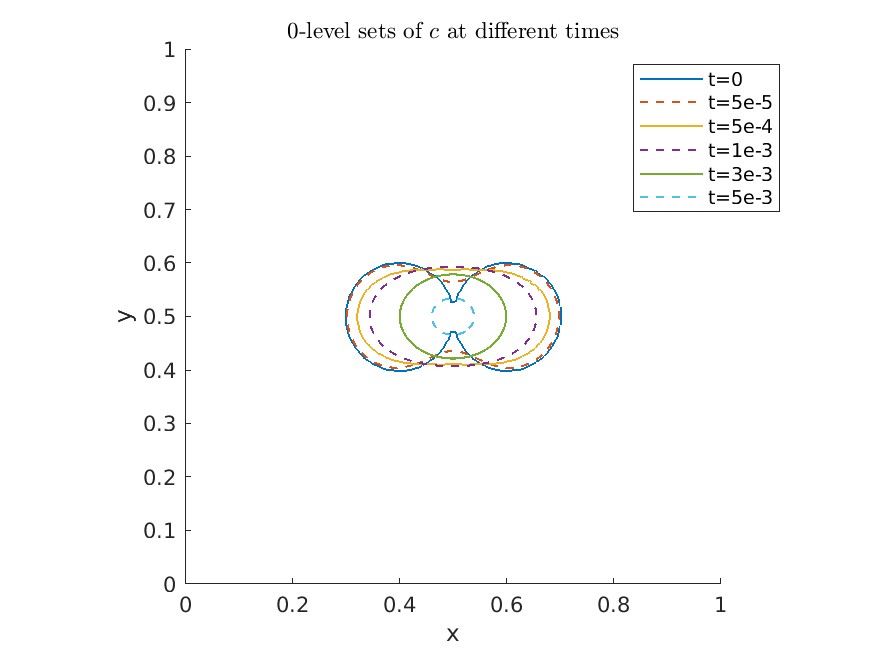}
    \caption{$0$-level sets of the $c$-variable at times $T=0, 5\cdot10^{-5}, 5\cdot10^{-4}, 1\cdot10^{-3}, 3\cdot10^{-3}, 5\cdot10^{-3}$.}
    \label{fig_test4_0level_c}
\end{figure}

Pressure fluctuations about their mean value and velocity field are shown in Figure \ref{fig_test4_vel_field}. 
It can be observed that vertexes appear around the merging bubbles. 
\begin{figure}[h]
    \centering
    \begin{minipage}[b]{0.45\textwidth}
        \centering
        \includegraphics[width=\textwidth]{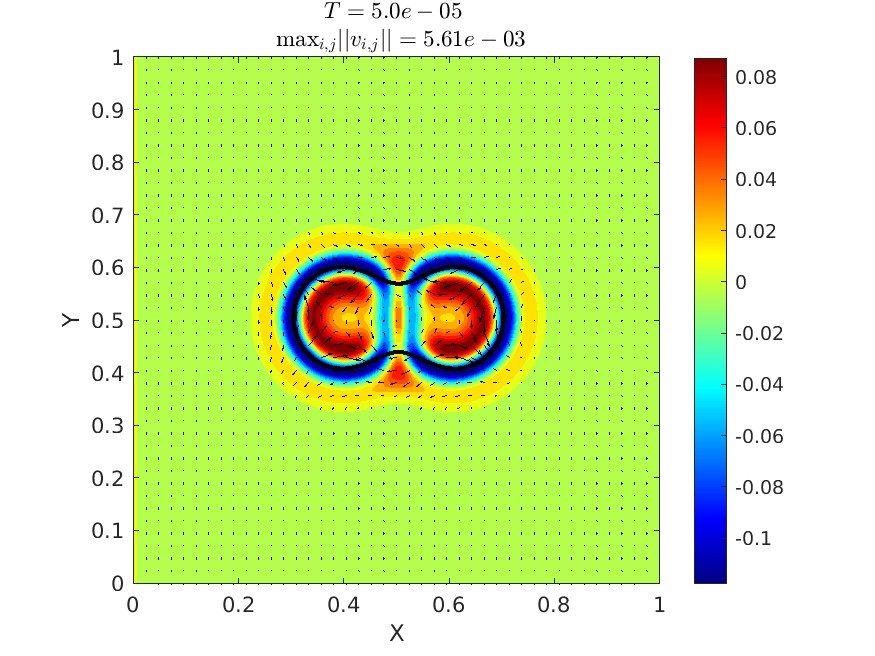}
    \end{minipage}
    \hfill
    \begin{minipage}[b]{0.45\textwidth}
        \centering
        \includegraphics[width=\textwidth]{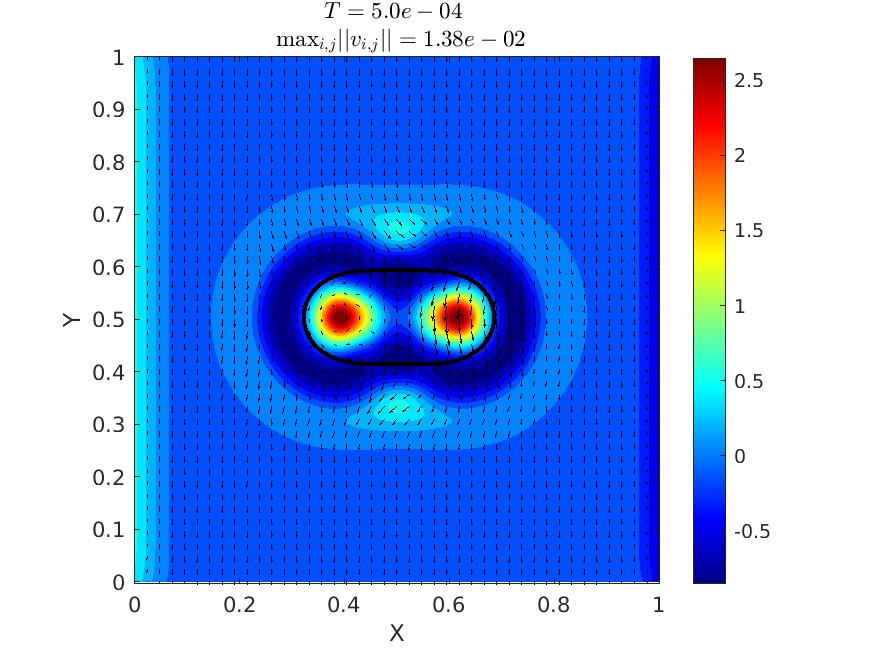}
    \end{minipage}
    \\[0.5em]
    \begin{minipage}[b]{0.45\textwidth}
        \centering
        \includegraphics[width=\textwidth]{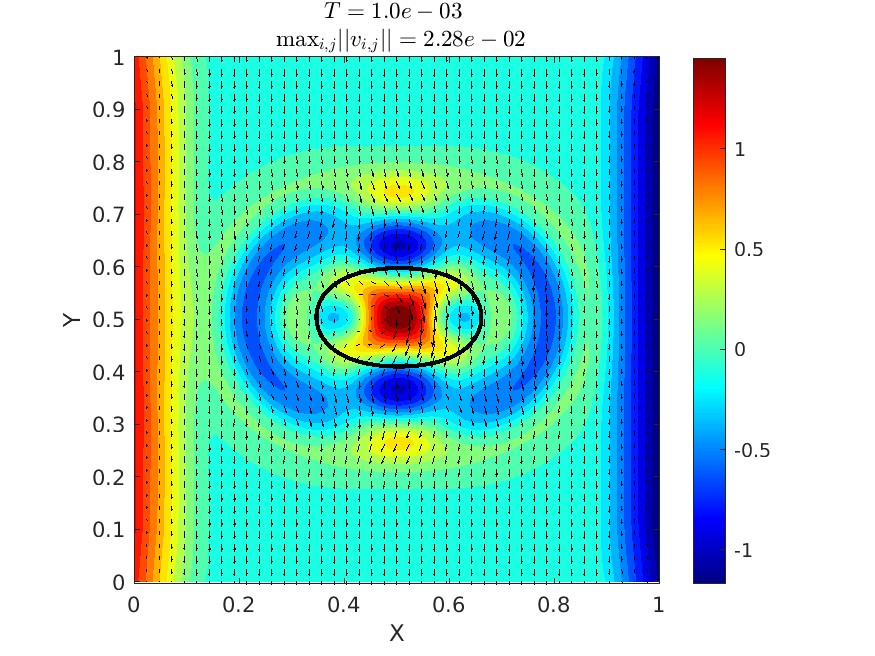}
    \end{minipage}
    \hfill
    \begin{minipage}[b]{0.45\textwidth}
        \centering
        \includegraphics[width=\textwidth]{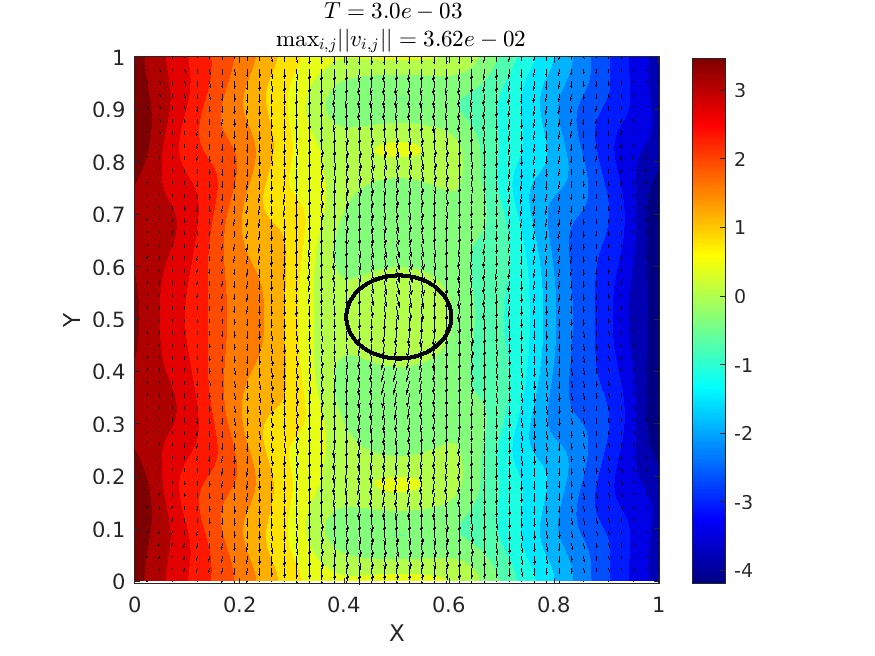}
    \end{minipage}
    \label{fig_test4_vel_field}
    \caption{Pressure fluctuations about their mean value and velocity field for times $T = 5\cdot10^{-5}, 5\cdot10^{-4}, 1\cdot10^{-3}, 3\cdot10^{-3}$.
    Vortexes are originated around the bubbles.}
\end{figure}

\section{Conclusions and future work}\label{section_conlusion}
In this work, we have developed an efficient second-order IMEX schemes
on staggered grids for the two-dimensional compressible isentropic
CHNS equations, with provably symmetric and positive definite matrices in the discretization of the viscosity terms.

Our numerical results show that this approach effectively captures the oscillations arising from high-order finite-differences discretizations of spatial derivatives. 
By treating the convective terms explicitly and the remaining terms implicitly, the time-step is constrained only by the convective subsystem.
Moreover, we have numerically verified second-order accuracy as well as bound preservation for both the density and the order parameter.

We point out that the time-step stability condition \eqref{time_step_cond} depends on the coefficient $C_{p}$. 
Consequently, for stiff pressure laws, where $C_p$ takes large values, the method is constrained to take smaller time steps.
To overcome this issue, we plan to extend the present framework to asymptotic-preserving IMEX schemes in low Mach number regimes \cite{Alazard06,LM98}.

We also intend to generalize the current approach to the three-dimensional case using Galerkin techniques, as well as to the quasi-incompressible CHNS model \cite{LT98}.

When approximating $C^{(i)}$ by the linear system derived from the CH type equation \eqref{eq_system_c}, a conjugate gradient method can be employed \cite{mulet_24}.
We plan to incorporate multigrid techniques as preconditioners for conjugate gradient solvers to further improve efficiency.

We aim to investigate high-order reconstructions for the convective terms to ensure both positivity and bound preservation for the density and the order parameter, respectively.

\subsection*{Conflict of interest} The authors declare that they have no conflict of interest.

\subsection*{Data Availability Statements} Data sharing is not applicable to this article as no datasets were generated or analyzed during the current study.

\vspace{-0.25cm}
\section*{Acknowledgments and competing interests}
\vspace{-0.15cm}
This paper has received financial support from the research projects
PID2023-146836NB-I00, granted by MCIN/ AEI /10.13039/ 501100011033, and 
\newline
CIAICO/2024/089, granted by GVA.
\vspace{-0.25cm}

\section*{References}
\bibliographystyle{plain}

\end{document}